\documentclass[11pt]{amsart}

\usepackage[margin=1.12in]{geometry}
\usepackage{amsmath,amssymb,amsfonts,mathtools}
\usepackage{bm}
\usepackage{microtype}
\usepackage{graphicx,placeins}
\usepackage{xcolor}
\usepackage{booktabs}
\usepackage{aliascnt}
\usepackage[hidelinks]{hyperref}

\usepackage[capitalize,noabbrev]{cleveref}
\hypersetup{
  pdftitle={Concentration of bounded sparse chaoses and sparse Khatri-Rao embeddings},
  pdfauthor={Guozheng Dai, Yiyun He, Ke Wang, Yizhe Zhu},
  pdfsubject={Moment inequalities for bounded sparse polynomial chaoses},
  pdfkeywords={polynomial chaos, Bennett inequality, sparse random variable, Khatri-Rao embedding, tensor sketching}
}

\allowdisplaybreaks[3]
\newtheorem{theorem}{Theorem}[section]
\newaliascnt{proposition}{theorem}
\newtheorem{proposition}[proposition]{Proposition}
\aliascntresetthe{proposition}
\newaliascnt{corollary}{theorem}
\newtheorem{corollary}[corollary]{Corollary}
\aliascntresetthe{corollary}
\newaliascnt{lemma}{theorem}
\newtheorem{lemma}[lemma]{Lemma}
\aliascntresetthe{lemma}
\theoremstyle{remark}
\newaliascnt{remark}{theorem}
\newtheorem{remark}[remark]{Remark}
\aliascntresetthe{remark}

\newcommand{\E}{\mathbb E}
\newcommand{\Pp}{\mathbb P}
\newcommand{\cJ}{\mathcal J}
\newcommand{\cS}{\mathcal S}
\newcommand{\norm}[1]{\left\lVert #1\right\rVert}
\newcommand{\abs}[1]{\left\lvert #1\right\rvert}
\newcommand{\ind}{\mathbf 1}
\newcommand{\Ber}{\operatorname{Bernoulli}}

\definecolor{previousrevision}{named}{black}
\DeclareRobustCommand{\rev}[1]{{\color{previousrevision}#1}}
\newenvironment{revision}{\par\begingroup\color{previousrevision}}{\par\endgroup}
\definecolor{currentrevision}{named}{black}
\DeclareRobustCommand{\edit}[1]{{\color{currentrevision}#1}}
\newenvironment{editing}{\par\begingroup\color{currentrevision}}{\par\endgroup}
\definecolor{reviewrevision}{named}{black}
\DeclareRobustCommand{\reviewedit}[1]{{\color{reviewrevision}#1}}
\DeclareRobustCommand{\appreview}[1]{{\color{black}#1}}
\title[Bounded sparse chaoses and Khatri--Rao embeddings]
{Concentration of bounded sparse chaoses and sparse Khatri--Rao embeddings}
\author{Guozheng Dai}
\address{Department of Mathematics, Hong Kong University of Science and
Technology, Clear Water Bay, Kowloon, Hong Kong}
\email{guozhengdai@ust.hk}

\author{Yiyun He}
\address{Department of Mathematics, University of California, San Diego, La Jolla, CA 92093 USA}
\email{yih130@ucsd.edu}

\author{Ke Wang}
\address{Department of Mathematics, Hong Kong University of Science and
Technology, Clear Water Bay, Kowloon, Hong Kong}
\email{kewang@ust.hk}

\author{Yizhe Zhu}
\address{Department of Mathematics, University of Southern California,
3620 Vermont Avenue, Los Angeles, CA 90089, USA}
\email{yizhezhu@usc.edu}
\date{}

\begin{document}
\flushbottom
\emergencystretch=2em
\setlength{\jot}{2pt}
\setlength{\abovedisplayskip}{7pt plus 2pt minus 3pt}
\setlength{\belowdisplayskip}{7pt plus 2pt minus 3pt}
\setlength{\abovedisplayshortskip}{4pt plus 2pt minus 2pt}
\setlength{\belowdisplayshortskip}{4pt plus 2pt minus 2pt}

\begin{abstract}
\rev{We prove moment and mixed-tail inequalities for decoupled homogeneous
chaoses of arbitrary fixed order generated by independent, centered,
bounded random variables with coordinate-dependent second-moment
bounds. The estimates hold for arbitrary real rectangular coefficient
tensors and are expressed through weighted slice and partition norms.
\edit{A Bennett-type logarithmic correction sharpens the contribution
of the largest coefficient to these bounds.}
\edit{The proof combines a comparison with sums of sparse exponential
variables with an elementary replication argument.}
\edit{We also extend the results to diagonal-free non-decoupled chaoses
and sparse sub-exponential inputs.}
\edit{As an application, we establish conditions under which sparse
Khatri--Rao random projections approximately preserve norms and inner
products for finite collections of tensors. For a sparse random relation
containing a planted dense block, we prove matching upper and lower bounds
on the required projection dimension, demonstrating the sharpness of the
logarithmic correction.}}
\end{abstract}

\maketitle

\section{\rev{Introduction and main results}}
\label{sec:intro}

Concentration inequalities quantify how far a random quantity can
deviate from its mean or another typical value.  They provide tools
for controlling estimation error, studying random matrices, and
choosing sample sizes in randomized algorithms
\cite{Ledoux2001,BLMbook,vershynin2018high}.
For nonlinear statistics, the useful bounds must reflect both the
input distributions and the structure of the function being studied.
\edit{Polynomials of independent variables arise in subgraph counts in
random graphs and in many estimators built from random vectors.
For concentration and large deviations of random graph polynomials,
see, for example,
\cite{kim2000concentration,chatterjee2012missing,chatterjee2016introduction,LZ17}.}

\begin{editing}
Our main object is a polynomial of the form
\begin{equation}\label{def:Z}
 Z\coloneqq\sum_{i_1=1}^{n_1}\cdots\sum_{i_d=1}^{n_d}
 a_{i_1,\ldots,i_d}X_{i_1}^{(1)}\cdots X_{i_d}^{(d)},
\end{equation}
where the coefficients \(a_{i_1,\ldots,i_d}\) are arbitrary real numbers
and the input variables \((X_i^{(s)})\) are independent, centered,
and bounded in absolute value by
one. Each term contains exactly \(d\) factors, one from each of \(d\)
independent groups of variables. Such a polynomial is called a
\emph{decoupled homogeneous chaos of order \(d\)}. For \(d=1\) this
is a sum; for \(d=2\) it is a bilinear form in two independent random
vectors. The coefficients form a multidimensional array, or tensor.
We seek bounds on the moments of \(Z\) and on the probability that
\(\abs{Z}\) is large.

We focus on sparse inputs, which are usually zero. For example, an
input taking values \(\pm1\) with probability \(p/2\) each and zero
otherwise has second moment \(p\). Our bounds allow general bounded
inputs with individual second-moment bounds, capturing information
that bounds based only on their maximum size can miss.

Two distinct gains arise for a Bernoulli sum.
Let \(S\coloneqq\sum_{i=1}^n(\delta_i-p)\), where the \(\delta_i\) are
independent Bernoulli variables with success probability \(0<p<1\),
and write \(v\coloneqq np(1-p)\) for its variance.
Hoeffding's inequality~\cite{Hoeffding1963} and Bennett's
inequality~\cite{bennett1962probability} give, respectively,
\[
 \Pp(S\geq t)\leq\exp(-2t^2/n),\qquad
 \Pp(S\geq t)\leq\exp(-v h(t/v)),\qquad t>0,
\]
where \(h(u)\coloneqq(1+u)\log(1+u)-u\).
For small \(p\) and \(t\) comparable to \(np\), the two exponents
have orders \(np^2\) and \(np\). This \emph{sparsity gain} comes
from using the variance and is also captured by Bernstein's inequality.
For larger deviations, \(t\gg v\), Bennett's exponent has order
\(t\log(t/v)\), whereas Bernstein's has order \(t\).
This additional logarithmic factor is the \emph{Bennett-type improvement}
we seek for higher-degree polynomials.
\end{editing}

We first place this question in the existing literature. Much of the theory of polynomial concentration proceeds through
moment estimates, which control averages of powers of the random
quantity and yield tail bounds through Markov's inequality. For sums of independent variables,
Montgomery-Smith~\cite{MS_PAMS} obtained tail estimates in the
Rademacher case. Gluskin and Kwapie{\'n}~\cite{gluskin1995tail}
and Hitczenko, Montgomery-Smith, and Oleszkiewicz~\cite{HMO97}
obtained sharp estimates for symmetric summands with log-concave
and log-convex tails, respectively. More generally,
Lata{\l}a~\cite{Latala97} established sharp moment estimates for
sums of independent nonnegative or symmetric random variables.
For quadratic forms in independent centered sub-Gaussian variables,
the Hanson--Wright inequality \cite{HW71,RV13} expresses concentration
in terms of the Frobenius and operator norms of the coefficient
matrix. Further results on the concentration of quadratic forms
can be found in \cite{Latala99,VW15,Adamczak15,KZ20,CY21}.

Beyond quadratic forms, Lata{\l}a~\cite{Latala2006} obtained
sharp moment estimates for Gaussian chaoses of arbitrary order,
using norms that measure the coefficient tensor along different
groupings of its indices. For symmetric inputs with
log-concave tails, Adamczak and Lata{\l}a~\cite{AL12} obtained
sharp bounds for chaoses of order at most three and for exponential
chaoses of arbitrary order. Kolesko and Lata{\l}a~\cite{KL15}
developed the corresponding moment theory for log-convex tails.
For general polynomials, Adamczak and Wolff~\cite{AW15}
developed bounds in terms of arrays of partial derivatives, and G\"otze,
Sambale, and Sinulis~\cite{gotze2021concentration} obtained
multilevel concentration inequalities for independent
\(\alpha\)-sub-exponential inputs.
Related moment inequalities for \(U\)-statistics and general
functions of independent variables can be found in
\cite{GLZ00,BBLM05,schudy2012concentration}.
Concentration inequalities for simple
random tensors were established by
Vershynin~\cite{vershynin2020concentration} and Bamberger,
Krahmer, and Ward~\cite{BKW22}. Adamczak, Lata{\l}a, and
Meller~\cite{ALM21} obtained moment and tail estimates for
Banach-valued Gaussian chaoses.

For sparse polynomials, Dai and Wang~\cite{DaiWang2026} obtained
moment and tail bounds that explicitly account for the probabilities
that the inputs are nonzero. Their results allow random amplitudes
with \(\alpha\)-sub-exponential tails, \(0<\alpha\leq1\), and provide
the starting point for the present work. We use the additional
assumption of bounded inputs to obtain a Bennett-type logarithmic
improvement.

For quadratic forms, Zhou~\cite{Zhou19} initiated the study of
the Hanson--Wright inequality for sparse sub-Gaussian random
vectors. He, Wang, and Zhu~\cite{HWZ24} established sparse
Hanson--Wright inequalities and obtained a Bennett-type
logarithmic gain through a combinatorial moment argument.
Further results on sparse quadratic and bilinear forms and
related missing-data models can be found in
\cite{Zhou20,PWL23,DHWZ_2025}.

\begin{editing}
We establish moment and tail bounds for the bounded decoupled
chaoses described above, of any fixed order
(\cref{thm:main,cor:main-tail}). Our bounds sharpen the contribution
of the largest coefficient by a Bennett-type logarithmic factor.
For sparse inputs given by Bernoulli variables times independent
bounded amplitudes, the other terms match those of Dai and
Wang~\cite{DaiWang2026}, up to constants depending only on the order.
We also give simpler formulations and explicit quadratic and cubic
cases (\cref{cor:gamma-form,cor:quadratic-moment,cor:rectangular-cubic}).
The proof uses a comparison argument and an elementary replication
construction; we explain the main idea at the end of \cref{sec:intro}.
\end{editing}

%\medskip
%\noindent\rev{\textbf{Application.}}
\begin{editing}
We apply our concentration bounds to sparse Khatri--Rao random
projections, which compress tensor data while approximately preserving
norms and inner products. To explain the construction, consider an
order-\(d\) tensor with \(n\) entries along each dimension, hence
\(n^d\) entries in total. An unstructured projection into
\(\mathbb R^M\) uses an \(M\times n^d\) matrix, which can be
expensive to store and apply. A Khatri--Rao map instead builds each
row as the tensor product of \(d\) vectors of length \(n\).
Each row is stored implicitly through these factors, requiring
\(dn\) numbers rather than \(n^d\). Each projected coordinate is
a polynomial of the form studied above, with the input tensor
supplying its coefficients.

This structure also speeds up computation when the input is itself
a tensor product of \(d\) vectors, called a rank-one tensor.
Given these input factors, each projected coordinate can be computed
as a product of \(d\) inner products of length-\(n\) vectors,
without forming the full tensor. Computing all \(M\) coordinates
therefore costs \(O(Mdn)\) operations. Making the projection factors
sparse reduces both storage and computation further: if each factor
entry is nonzero with probability \(p\), the expected computation
cost becomes \(O(Md(np+1))\). We study how this sparsity and the
structure of the input affect the projection dimension \(M\) needed
to approximately preserve norms and pairwise inner products.
\end{editing}

Sparse Khatri--Rao sketches have been studied previously.
Sun, Guo, Tropp, and Udell~\cite{SunGuoTroppUdell2021} introduce
tensor random projections and discuss sparse sign factors.
Chen and Jin~\cite{ChenJin2021} analyze normalized products of
Bernoulli and sub-Gaussian variables, including the two-factor case of the sparse-sign maps considered in this paper.  They prove finite-set,
subspace, and constrained least-squares guarantees.
Rakhshan and Rabusseau~\cite{RakhshanRabusseau2020} study Gaussian
projections in rank-one, sums-of-rank-one, and tensor-train forms.
Ahle and Knudsen~\cite{AhleKnudsen2019} establish strong
Johnson--Lindenstrauss moment guarantees for tensor-product sketches.
\edit{For subspace embeddings, Bujanovi{\'c}, Grubi{\v{s}}i{\'c},
Kressner, and Lam~\cite{BujanovicGrubisicKressnerLam2026} study
maps whose rows are tensor products of two independent Gaussian
vectors. Beretta and Musco~\cite{BerettaMusco2026} consider tensor
products of any fixed number of independent isotropic random vectors
with uniformly bounded sub-Gaussian norm, and obtain a required
projection dimension that grows nearly linearly with the dimension
of the subspace.}

\edit{Our application identifies conditions on the arrangement of
nonzero tensor entries and the sparsity of the projection under which
the Bennett-type logarithmic correction reduces the projection
dimension required by our bounds. Under these conditions, a single
random map approximately preserves the norms of a finite collection
of tensors and the inner products between their normalized versions,
simultaneously with high probability. For tensors consisting of a block of ones surrounded by a sparse
random background of zeros and ones, we prove matching upper and
lower bounds on the projection dimension required for norm preservation
by our map. These bounds show that the logarithmic improvement is
essential for this family.}

Beyond randomized numerical linear algebra, we expect our inequalities
to be useful for studying subgraph counts in sparse random graphs and hypergraphs,
a classical setting for polynomial concentration and upper-tail
problems~\cite{JR02,demarco2012tight,HMS22}.
We leave these developments to future work.

\begin{editing}
The rest of the paper is organized as follows.
In the remainder of \cref{sec:intro}, we state the main concentration
bounds, give their quadratic and cubic forms, and explain the proof idea.
\Cref{sec:applications} states the
application results, \cref{sec:proof-main-results} proves the main
concentration inequalities, and \cref{sec:application-proofs} gives
the application proofs. \appreview{Extensions beyond the bounded decoupled
setting are given in \cref{app:extensions}, and general embedding bounds
in Appendix~\ref{sec:general-sketches}.}
\end{editing}

\medskip
\noindent\rev{\textbf{Notation.}}
For an integer \(m\geq1\), write \([m]\coloneqq\{1,\ldots,m\}\), and adopt the
convention that empty products equal one.  For real numbers \(a,b\),
write \(a\vee b\coloneqq\max\{a,b\}\).  Unless otherwise specified,
the letters \(C,c>0\) denote constants whose values may change from line
to line; subscripts indicate their permitted parameter dependence.  We
write \(f\lesssim_\alpha g\) if \(f\leq C_\alpha g\),
\(f\asymp_\alpha g\) if both comparison inequalities hold, and use
\(O_\alpha(\cdot)\) with the same convention; omitted subscripts indicate
universal constants.  The notation \(f\ll g\) means that \(f/g\) is
sufficiently small in the stated regime.  For a finite tensor
\(T=(t_\alpha)\), set \(\norm{T}_{\mathrm F}
 \coloneqq\bigl(\sum_\alpha \abs{t_\alpha}^2\bigr)^{1/2}\) and
 \(\norm{T}_\infty\coloneqq\max_\alpha\abs{t_\alpha}\).
For a matrix \(B=(b_{ij})\), let \(\norm{B}_{\mathrm{op}}\) denote its
Euclidean operator norm and set
\(\norm{B}_{2,\infty}
 \coloneqq\max_i\bigl(\sum_j\abs{b_{ij}}^2\bigr)^{1/2}\).
For a random variable \(Y\) and \(r\geq1\), write
\(\norm{Y}_{L_r}\coloneqq\left(\E\abs{Y}^r\right)^{1/r}\).
Throughout the paper, \(\ell(x)\coloneqq1+\log x\) for \(x\geq1\).
\appreview{Throughout, we assume that the coefficient tensor \(A\) is nonzero.}

\subsection{\rev{Bennett-type inequalities for chaoses}}
\label{sec:chaos}
\reviewedit{We first specify the input assumptions and the tensor norms needed
for the bounds. We then state the general results, explain their
low-order cases, and outline the proof.}

Fix an order \(d\geq1\) and dimensions \(n_1,\ldots,n_d\), and write
\[
\bm i\coloneqq(i_1,\ldots,i_d),
\qquad
a_{\bm i}\coloneqq a_{i_1,\ldots,i_d}.
\]
\edit{Recall the decoupled homogeneous polynomial chaos of order \(d\)
defined in \eqref{def:Z}, which we write as}
\[
Z_A(X)
\coloneqq
\sum_{\bm i\in[n_1]\times\cdots\times[n_d]}
a_{\bm i}\prod_{s=1}^dX_{i_s}^{(s)}.
\]
Here \(A=(a_{\bm i})\) is a deterministic rectangular \(d\)-tensor,
and \(X=(X_i^{(s)})_{s\in[d],\,i\in[n_s]}\) is a mutually independent
family satisfying
\begin{align}\label{eq:assump-X}
\E X_i^{(s)}=0,
\qquad
\abs{X_i^{(s)}}\leq1\ \text{a.s.},
\qquad
\E\bigl(X_i^{(s)}\bigr)^2\leq p_i^{(s)},
\qquad
0<p_i^{(s)}\leq1.
\end{align}
\reviewedit{As described above, \emph{homogeneous} refers to the common degree
\(d\), and \emph{decoupled} to the use of independent input vectors.}  For example, \(\sum_{i,j}a_{ij}X_i^{(1)}X_j^{(2)}\)
is a decoupled quadratic form.  The numbers \(p_i^{(s)}\) are upper
bounds on second moments. \reviewedit{They may vary from one coordinate
to another, allowing the bounds to reflect different degrees of sparsity.}
A basic example is
\(X_i^{(s)}=\delta_i^{(s)}\sigma_i^{(s)}\), where
\(\delta_i^{(s)}\sim\Ber(p_i^{(s)})\) are Bernoulli random variables
and \(\sigma_i^{(s)}\) are Rademacher random variables.

\reviewedit{The bounds depend on both the sizes of the coefficients and their
arrangement within the tensor. We describe these features using norms
of slices obtained by fixing some indices, following
\cite{DaiWang2026,gotze2021concentration}.}  For \(I\subseteq[d]\), write
\(I^c\coloneqq[d]\setminus I\) and \(\bm i_I\coloneqq(i_s)_{s\in I}\). For a given $d$-tensor $A$, we use
\[
 A_{\bm i_{I^c}}\coloneqq(a_{\bm i})_{\bm i_{I^c}}
\]
for the \emph{slice} obtained by fixing the coordinates in \(I\) and allowing
those in \(I^c\) to vary.  Thus the subscript records the free
coordinates; the fixed coordinates are understood from context.
\reviewedit{For a three-dimensional tensor, fixing the first index leaves a
matrix indexed by the other two; fixing the first two leaves a vector
indexed by the third.}
When \(I=\varnothing\), the slice is the full tensor \(A\), and the
maximum over fixed indices is omitted.  For any finite set \(S\), let \(\Pi(S)\) be the set of partitions of
\(S\), with the convention that
\(\Pi(\varnothing)\coloneqq\{\varnothing\}\).

Let \(U\subseteq[d]\) and \(\cJ=\{J_1,\ldots,J_m\}\in\Pi(U)\), and write
\(\abs{\cJ}\coloneqq m\) for its number of blocks.
\reviewedit{For an array \(T=(t_{\bm i_U})\) indexed by the coordinates in
\(U\), the \emph{partition norm associated with \(\cJ\)} is}
\[
\norm{T}_{\cJ}
\coloneqq
\sup\left\{
\sum_{\bm i_U}t_{\bm i_U}
\prod_{q=1}^m x^{(q)}_{\bm i_{J_q}}:
\norm{x^{(q)}}_2\leq1,\ q\in[m]
\right\}.
\]
For \(U=\varnothing\), this norm is the absolute value of the scalar.

The one-block and singleton partitions give the Frobenius and injective
multilinear norms, respectively.
For a matrix, these are precisely
\[
 \norm A_{\{\{1,2\}\}}=\norm A_{\mathrm F},
 \qquad
 \norm A_{\{\{1\},\{2\}\}}=\norm A_{\mathrm{op}}.
\]
More generally, a two-block partition is an operator norm after
reshaping the tensor as a matrix.  Partition norms thus extend familiar
matrix norms by specifying which indices are grouped together.

\reviewedit{We now incorporate the second-moment bounds into these norms.
The weighted slice norm measures the size of a slice after weighting
its entries; the moment scale also includes the powers of the moment
order needed in the bound.}
For \(I\subseteq[d]\) and \(\cJ\in\Pi(I^c)\), the \emph{weighted slice
norm} \(\cS_{I,\cJ}(A)\) and, for \(r\geq2\), the associated
\emph{moment scale} \(\Delta_I(A;r)\) are
\begin{equation}\label{eq:Delta}
\begin{aligned}
\cS_{I,\cJ}(A)
&\coloneqq
\max_{\bm i_I}
\left\|
\left(
a_{\bm i}
\prod_{s\in I^c}\sqrt{p_{i_s}^{(s)}}
\right)_{\bm i_{I^c}}
\right\|_{\cJ},\\
\Delta_I(A;r)
&\coloneqq
\max_{\cJ\in\Pi(I^c)}
r^{\,\kappa_{I,\cJ}}\,\cS_{I,\cJ}(A),
\qquad
\kappa_{I,\cJ}\coloneqq\abs I+\abs{\cJ}/2.
\end{aligned}
\end{equation}

We omit the dependence on the variance proxies \(p_i^{(s)}\) from the notation.
In particular, \(\cS_{[d],\varnothing}(A)=\norm{A}_\infty\), and
\(\Delta_{[d]}(A;r)=r^d\norm{A}_\infty\) has the highest order in \(r\).

\reviewedit{To compute the weighted slice norm, fix the coordinates in \(I\),
weight each free coordinate by its standard-deviation bound
\(\sqrt{p_i^{(s)}}\), and take the partition norm of the resulting
slice. Finally, maximize over the choices of fixed coordinates.}
Each fixed coordinate contributes a
factor \(r\) to the moment estimate; each block of free coordinates
contributes \(\sqrt r\).  For a matrix and a common variance bound
\(p\) for all input coordinates, the four resulting scales are
\[
 \sqrt r\,p\norm A_{\mathrm F},\qquad
 rp\norm A_{\mathrm{op}},\qquad
 r^{3/2}\sqrt p\max\left\{\norm A_{2,\infty},
                            \norm{A^{\mathsf T}}_{2,\infty}\right\},\qquad
 r^2\norm A_\infty.
\]
Only the last scale receives the logarithmic correction in the theorem.

\appreview{Write \(m_A\coloneqq\norm A_\infty\).}
\reviewedit{To state the theorem compactly, we group the moment scales by
the number of free coordinates. For \(0\leq k\leq d\), let
\(B_k(A;r)\) be the largest scale with \(k\) free coordinates, and
let \(\lambda_A(r)\) compare these grouped scales with the largest-entry
contribution:}
\begin{equation}
B_k(A;r)\coloneqq\max_{\substack{I\subseteq[d]\\\abs{I^c}=k}}\Delta_I(A;r),
 \qquad
 \lambda_A(r)\coloneqq1\vee\min_{1\leq k\leq d}
 \left(\frac{r^dm_A}{B_k(A;r)}\right)^{2/k}.
\end{equation}
Then \(B_0(A;r) = \Delta_{[d]}(A;r) = r^dm_A\). The quantity
\(\lambda_A(r)\) measures how far the largest-entry
contribution \(r^dm_A\) exceeds the other moment scales. Recall that \(\ell(x)\coloneqq1+\log x\) for \(x\geq1\).

\begin{theorem}[Bennett-type moment bound]
\label{thm:main}
For $d\ge 1$, let \(A\) be a \appreview{real $d$-tensor}, and let the mutually independent family
\(X=(X_i^{(s)})_{s\in[d],\,i\in[n_s]}\) satisfy \eqref{eq:assump-X}.  For every \(r\geq2\),
\begin{equation}\label{eq:explicit-main-moment}
 \norm{Z_A(X)}_{L_r}
 \lesssim_d
 \max_{1\leq k\leq d}B_k(A;r)
 +\frac{r^dm_A}{\ell(\lambda_A(r))^d}.
\end{equation}
\end{theorem}

The first term of \eqref{eq:explicit-main-moment} collects the usual
weighted slice and partition scales.  When \(\lambda_A(r)>1\), the
largest-entry contribution dominates, and the bound simplifies to
\begin{equation}\label{eq:effective-main-moment}
 \norm{Z_A(X)}_{L_r}
 \lesssim_d\frac{r^dm_A}{\ell(\lambda_A(r))^d}.
\end{equation}
Indeed, \(\lambda_A(r)>1\) implies \(B_k(A;r)\leq r^dm_A\lambda_A(r)^{-k/2}\), and these terms
are absorbed into the right-hand side up to a constant depending only
on \(d\).  The gain over \(r^dm_A\) grows with \(\lambda_A(r)\).
For Bernoulli random variables times independent bounded amplitudes,
the terms \(B_k(A;r)\), \(1\leq k\leq d\),
match the corresponding scales in the moment estimates of Dai and
Wang~\cite{DaiWang2026}, up to constants depending only on \(d\).
The new feature of \cref{thm:main} is the logarithmic improvement of
the largest-entry contribution, replacing \(r^dm_A\) by
\(r^dm_A/\ell(\lambda_A(r))^d\), without changing the other scales.

\reviewedit{The moment estimates also give the following tail bound.
Here \(\eta_A(t)\) plays the role of \(\lambda_A(r)\) at a specified
deviation level. It controls the logarithmic improvement in the tail
exponent.}

\begin{corollary}[Bennett-type tail bound]
\label{cor:main-tail}
Under the assumptions of \cref{thm:main}, for \(t>0\) set
\begin{equation}\label{eq:explicit-tail-separation}
 \eta_A(t)\coloneqq1\vee
 \min_{\substack{I\subsetneq[d]\\\cJ\in\Pi(I^c)}}
 \left[
  \frac{m_A}{\cS_{I,\cJ}(A)}
  \left(\frac{t}{m_A}\right)^{1-\kappa_{I,\cJ}/d}
 \right]^{2/\abs{I^c}}.
\end{equation}
Then
\begin{equation}\label{eq:explicit-mixed-tail}
 \Pp\left(\abs{Z_A(X)}\geq t\right)
 \leq2\exp\left[-c_d\min\left\{
 \min_{\substack{I\subsetneq[d]\\\cJ\in\Pi(I^c)}}
 \left(\frac{t}{\cS_{I,\cJ}(A)}\right)^{1/\kappa_{I,\cJ}},\,
 \left(\frac{t}{m_A}\right)^{1/d}\ell(\eta_A(t))
 \right\}\right].
\end{equation}
\end{corollary}

The logarithm in \eqref{eq:explicit-mixed-tail}
improves the largest-entry rate \((t/m_A)^{1/d}\) when
\(\eta_A(t)\) is large.

\medskip
\reviewedit{\noindent\textbf{Simplified bounds via a moment threshold.}}
\reviewedit{The parameters \(\lambda_A(r)\) and \(\eta_A(t)\) depend on the
moment order \(r\) and deviation level \(t\), respectively. A simpler
form of the bounds uses a single threshold determined by the
coefficient tensor and second-moment bounds. Recall
\(\cS_{I,\cJ}(A)\) and \(\kappa_{I,\cJ}\) from \eqref{eq:Delta}.
\appreview{Define}}
\begin{equation}\label{eq:gamma-threshold}
 \gamma_d(A)\coloneqq\max_{\substack{I\subsetneq[d]\\\cJ\in\Pi(I^c)}}
 \left(\frac{\cS_{I,\cJ}(A)}{m_A}\right)^{1/(d-\kappa_{I,\cJ})}.
\end{equation}

\reviewedit{This threshold is well defined and positive since
\(d-\kappa_{I,\cJ}=\abs{I^c}-\abs\cJ/2\geq\abs{I^c}/2>0\) for
\(I\subsetneq[d]\), and all \(p_i^{(s)}>0\).
The term \(r^dm_A\) dominates all the other moment scales exactly
when \(r\geq\gamma_d(A)\). The following corollary expresses the
logarithmic gain in terms of this threshold.}

\begin{corollary}[\rev{Bennett bounds with a fixed threshold}]
\label{cor:gamma-form}
Under the assumptions of \cref{thm:main}, for every
\(r\geq\max\{2,e\,\gamma_d(A)\}\),
\begin{equation}\label{eq:gamma-moment}
 \norm{Z_A(X)}_{L_r}
 \lesssim_d\frac{r^dm_A}{[1+\log(r/\gamma_d(A))]^d}.
\end{equation}
Moreover, for every \(t>0\),
\begin{equation}\label{eq:gamma-tail}
\begin{aligned}
 \Pp\left(\abs{Z_A(X)}\geq t\right)
 &\leq2\exp\Bigg[-c_d\min\Bigg\{
 \min_{\substack{I\subsetneq[d]\\\cJ\in\Pi(I^c)}}
 \left(\frac{t}{\cS_{I,\cJ}(A)}\right)^{1/\kappa_{I,\cJ}},\\
 &\hspace{14mm}
 \left(\frac{t}{m_A}\right)^{1/d}
 \ell\left(1\vee\frac{(t/m_A)^{1/d}}{\gamma_d(A)}\right)
 \Bigg\}\Bigg].
\end{aligned}
\end{equation}
\end{corollary}

\reviewedit{The comparison with the original bounds follows from
\(\lambda_A(r)\geq r/\gamma_d(A)\) for \(r\geq\gamma_d(A)\), and
\(\eta_A(t)\geq1\vee(t/m_A)^{1/d}/\gamma_d(A)\) for all \(t>0\)
(see the proof of \cref{cor:gamma-form}).
Thus the threshold formulation may give a smaller logarithmic gain,
but is easier to use in the explicit quadratic and cubic bounds below.}
\medskip
\begin{revision}
\noindent\textbf{Extensions.}
Two extensions are recorded in \cref{app:extensions}. Standard
decoupling~\cite{deLaPenaMontgomerySmith1995} transfers the bounded-input
estimates, including the logarithmic gain, to diagonal-free
non-decoupled chaoses. This includes the quadratic form below.
For sparse \(\alpha\)-sub-exponential inputs with
\(1<\alpha\leq\infty\), the same comparison yields a logarithmic
power \(d\beta\) in the moment denominator and \(\beta\) in the
tail rate \appreview{in the logarithmic regime}, where \(\beta\coloneqq1-1/\alpha\). The choice
\(\alpha=\infty\) recovers the bounded case.
\end{revision}

\medskip
\reviewedit{We illustrate these results in the linear, quadratic, and cubic
cases. The linear case recovers a familiar Bennett bound; the quadratic
and cubic cases give explicit moment bounds where the logarithmic
correction is effective and tail bounds valid for every \(t>0\).}

\medskip
\noindent\textbf{The linear case ($d=1$).}
For \(S\coloneqq\sum_i a_iX_i\), consider \(V^2\coloneqq\sum_i a_i^2p_i>0\) and
\(a_{\max}\coloneqq\max_i\abs{a_i}\).  The theorem recovers, up to universal constants, the moment bound implied by Bennett’s inequality~\cite{bennett1962probability}:
\begin{equation}
 \norm S_{L_r}\lesssim\sqrt r\,V+
 \frac{ra_{\max}}{\ell(1\vee ra_{\max}^2/V^2)},\qquad r\geq2.
\end{equation}

\medskip
\noindent\textbf{The quadratic case ($d=2$).}
\reviewedit{Applying the decoupling extension to the quadratic form
\(X^{\mathsf T}AX\) gives the following sparse analogue of the
Hanson--Wright inequality~\cite{RV13} for bounded inputs, with a
Bennett-type logarithmic improvement.}
\begin{corollary}[Quadratic Bennett bounds]\label{cor:quadratic-moment}
Let $A\in \mathbb R^{n\times n}$ be a \appreview{symmetric matrix} with
zero diagonal entries.  Let \(X_1,\ldots,X_n\) be independent and centered random variables with
\(\abs{X_i}\leq1\) and \(\E X_i^2\leq p_i\in(0,1]\).
Set \(m_A\coloneqq\norm A_\infty\),
\(D_p\coloneqq\operatorname{diag}(\sqrt{p_1},\ldots,\sqrt{p_n})\), and
\[
 \gamma_2(A)\coloneqq\max\left\{
 \frac{\norm{AD_p}_{2,\infty}^2}{m_A^2},\,
 \left(\frac{\norm{D_pAD_p}_{\mathrm F}}{m_A}\right)^{2/3},\,
 \frac{\norm{D_pAD_p}_{\mathrm{op}}}{m_A}
 \right\}.
\]
For every \(r\geq\max\{2,e\,\gamma_2(A)\}\),
\begin{equation}\label{eq:quadratic-moment}
 \norm{X^{\mathsf T}AX}_{L_r}
 \lesssim\frac{r^2m_A}{[1+\log(r/\gamma_2(A))]^2}.
\end{equation}
Moreover, for every \(t>0\),
\begin{equation}\label{eq:quadratic-tail}
\begin{aligned}
 \Pp\left(\abs{X^{\mathsf T}AX}\geq t\right)
 &\leq2\exp\Bigg[-c\min\Bigg\{
 \frac{t^2}{\norm{D_pAD_p}_{\mathrm F}^2},\,
 \frac{t}{\norm{D_pAD_p}_{\mathrm{op}}},\\
 &\hspace{14mm}
 \left(\frac{t}{\norm{AD_p}_{2,\infty}}\right)^{2/3},\,
 \sqrt{\frac{t}{m_A}}\,
 \ell\left(1\vee\frac{\sqrt{t/m_A}}{\gamma_2(A)}\right)
 \Bigg\}\Bigg].
\end{aligned}
\end{equation}
\end{corollary}

\reviewedit{Here \(\gamma_2(A)\) is the threshold in
\eqref{eq:gamma-threshold} specialized to this symmetric quadratic setting.}

\reviewedit{Our quadratic bound removes the nonnegativity assumption on
\(A\) in He, Wang, and Zhu~\cite[Lemma~6.3]{HWZ24}.
In their setting of nonnegative \(A\) and inputs \(X_i=\delta_i\chi_i\),
with independent Bernoulli variables \(\delta_i\) and independent
centered amplitudes \(\chi_i\) bounded by one, our row and operator
terms also sharpen their weighted row-sum bounds. Writing
\(\gamma_{\mathrm{HWZ}}\coloneqq\max_i\sum_{j\neq i}a_{ij}p_j\), we have}
\[
 \norm{AD_p}_{2,\infty}^2\leq\norm A_\infty\gamma_{\mathrm{HWZ}},
 \qquad \norm{D_pAD_p}_{\mathrm{op}}\leq\gamma_{\mathrm{HWZ}}.
\]
\reviewedit{The first inequality follows from
\(a_{ij}^2\leq\norm A_\infty a_{ij}\); the second follows by similarity
to \(AD_p^2\) and its maximum row-sum bound. Thus the row and operator
terms in \(\gamma_2(A)\) are no larger than their HWZ counterparts
and can be strictly smaller, while the Frobenius term is unchanged.}

\medskip
\noindent\textbf{The cubic case ($d=3$).}
\reviewedit{We next specialize \cref{cor:gamma-form} to order three,
with equal dimensions and a common second-moment bound.
For \(A\in\mathbb R^{n\times n\times n}\), write}
\[
 \norm A_{\mathrm{inj}}
 \coloneqq\sup_{\norm x_2,\norm y_2,\norm z_2\leq1}
 \Big|\sum_{i,j,k}a_{ijk}x_i y_j z_k\Big|
\]
for its injective norm.

\reviewedit{The bounds involve six weighted norms of vector slices,
matrix slices, and the full tensor. In the notation below, the
subscript records the number of free indices. The quantities \(V\)
use Euclidean or Frobenius norms, \(W\) use operator norms, and
\(U_3\) uses the injective norm.}

\begin{corollary}[Cubic Bennett bounds]
\label{cor:rectangular-cubic}
Suppose that the assumptions of \cref{thm:main} hold with \(d=3\)
and \(n_1=n_2=n_3=n\), and that \(p_i^{(s)}=p\in(0,1]\)
for all \(s\in[3]\) and \(i\in[n]\).  Set
\(m_A\coloneqq\norm A_\infty\) and define
\[
\begin{aligned}
 V_1&\coloneqq\sqrt p\max_{s\in[3]}
       \max_{\bm i_{\{s\}^c}}\norm{A_{\bm i_{\{s\}}}}_2,
 &V_2&\coloneqq p\max_{s<u}
       \max_{\bm i_{\{s,u\}^c}}\norm{A_{\bm i_{\{s,u\}}}}_{\mathrm F},\\
 W_2&\coloneqq p\max_{s<u}
       \max_{\bm i_{\{s,u\}^c}}\norm{A_{\bm i_{\{s,u\}}}}_{\mathrm{op}},
 &V_3&\coloneqq p^{3/2}\norm A_{\mathrm F},\\
 W_3&\coloneqq p^{3/2}\max_{s\in[3]}
       \norm A_{\{\{s\},[3]\setminus\{s\}\}},
 &U_3&\coloneqq p^{3/2}\norm A_{\mathrm{inj}}.
\end{aligned}
\]
Define
\begin{equation}\label{eq:cubic-threshold}
 \gamma_3(A)\coloneqq\max\left\{
 \left(\frac{V_1}{m_A}\right)^2,\,
 \left(\frac{V_2}{m_A}\right)^{2/3},\,
 \frac{W_2}{m_A},\,
 \left(\frac{V_3}{m_A}\right)^{2/5},\,
 \left(\frac{W_3}{m_A}\right)^{1/2},\,
 \left(\frac{U_3}{m_A}\right)^{2/3}
 \right\}.
\end{equation}
For every \(r\geq\max\{2,e\,\gamma_3(A)\}\),
\begin{equation}\label{eq:rectangular-cubic-moment}
 \norm{Z_A(X)}_{L_r}
 \lesssim\frac{r^3m_A}{[1+\log(r/\gamma_3(A))]^3}.
\end{equation}
Moreover, for every \(t>0\),
\begin{equation}\label{eq:rectangular-cubic-tail}
\begin{aligned}
 \Pp\left(\abs{Z_A(X)}\geq t\right)
 &\leq2\exp\Bigg[-c\min\Bigg\{
 \left(\frac{t}{V_1}\right)^{2/5},\,
 \left(\frac{t}{V_2}\right)^{2/3},\,
 \left(\frac{t}{W_2}\right)^{1/2},\\
 &\hspace{14mm}
 \frac{t^2}{V_3^2},\,
 \frac{t}{W_3},\,
 \left(\frac{t}{U_3}\right)^{2/3},\,
 \left(\frac{t}{m_A}\right)^{1/3}
 \ell\left(1\vee\frac{(t/m_A)^{1/3}}{\gamma_3(A)}\right)
 \Bigg\}\Bigg].
\end{aligned}
\end{equation}
\end{corollary}

\reviewedit{In \(V_1\), the vector \(A_{\bm i_{\{s\}}}\) is obtained
by fixing the other two indices. In \(V_2\) and \(W_2\), the matrix
\(A_{\bm i_{\{s,u\}}}\) is obtained by fixing the remaining index.
For \(W_3\), the full tensor is reshaped into a matrix with index
\(s\) along the rows and the other two indices along the columns,
and we take the largest operator norm over the three choices of \(s\).}

\medskip
\noindent\rev{\textbf{Idea of the proof.}}
\reviewedit{The proof of \cref{thm:main} combines a moment comparison with
an exact calculation of how tensor norms change under replication.
After symmetrizing the bounded inputs, we compare each with a suitably
scaled sum of independent sparse exponential variables with random
signs. Their factorial moment growth allows a logarithmic reduction
in the scale of the comparison sum. Expanding the sums repeats the
tensor coefficients without changing their values, a construction we
call \emph{replication}. Known moment bounds for the resulting chaos
then reduce the proof to choosing how many copies to use.}

\noindent\rev{\textbf{Example of replication.}}
For a simple example, start with the linear form \(aX_1+bX_2\).
Suppose \(X_1,X_2\) are independent and centered, with
\(|X_j|\leq1\) and \(\E X_j^2\leq p_j\in(0,1]\).
For an integer \(N\geq1\), instead of studying $X_j$'s, we define the
reference variables by
\[
 Y_i\coloneqq\delta_i\sigma_iE_i,\qquad i=1,\ldots,2N,
\]
where \(\delta_1,\ldots,\delta_N\sim\Ber(p_1)\) and
\(\delta_{N+1},\ldots,\delta_{2N}\sim\Ber(p_2)\),
\(\sigma_i\)'s are Rademacher random variables,
and \(E_i\sim\operatorname{Exp}(1)\) have mean one.
All these factors are mutually independent.  Thus each \(Y_i\) is
zero or a randomly signed exponential random variable.
We compare \(aX_1+bX_2\) with
\[
 a\sum_{i=1}^N Y_i+b\sum_{i=N+1}^{2N}Y_i.
\]
The coefficient vector changes from \((a,b)\) to
\[
 \bigl(\underbrace{a,\ldots,a}_{N\text{ copies}},
       \underbrace{b,\ldots,b}_{N\text{ copies}}\bigr),
\]
where the largest absolute entry is unchanged but the Euclidean norm
grows by a factor of \(\sqrt N\).

\reviewedit{This example illustrates why replication treats the largest
coefficient differently from the other norms. To quantify the gain,
write \(S_1\coloneqq\sum_{i=1}^N Y_i\) and
\(S_2\coloneqq\sum_{i=N+1}^{2N}Y_i\).}
\reviewedit{We first compare moments and then estimate the replicated sum.}
For an even integer \(r\geq2\), the moment comparison
in \cref{lem:replicated-moment-comparison} gives
\[
 \norm{aX_1+bX_2}_{L_r}
 \leq\frac{2\sqrt e}{1+\log N}\,
       \norm{aS_1+bS_2}_{L_r}.
\]
Second, \(aS_1+bS_2\) is a sum of independent random variables with
variance \(2N(a^2p_1+b^2p_2)\).  The standard Bernstein moment
estimate, also the degree-one case of
\cref{lem:sparse-exponential-chaos}, gives
\[
 \norm{aS_1+bS_2}_{L_r}
 \leq C\left[\sqrt{Nr(a^2p_1+b^2p_2)}
                 +r\max\{|a|,|b|\}\right].
\]
Here the second term, which corresponds to the largest coefficient, is independent of $N$.
Thus we can choose the optimal number of copies $N$ by combining the two bounds:
\[
 \norm{aX_1+bX_2}_{L_r}
 \leq C\,
 \frac{\sqrt{Nr(a^2p_1+b^2p_2)}
       +r\max\{|a|,|b|\}}{1+\log N}.
\]
For \((a,b)\neq(0,0)\), the choice
\(N=\left\lceil
 \frac{r\max\{a^2,b^2\}}{a^2p_1+b^2p_2}
 \right\rceil\)
balances the two numerator terms up to a universal constant and gives
\[
 \norm{aX_1+bX_2}_{L_r}
 \leq
 \frac{C r\max\{|a|,|b|\}}
 {1+\log\!\left(
     \dfrac{r\max\{a^2,b^2\}}{a^2p_1+b^2p_2}
    \right)},\qquad r\geq2.
\]
\reviewedit{
The argument of the logarithm is at least one because
\(a^2p_1+b^2p_2\leq2\max\{a^2,b^2\}\) and \(r\geq2\).
For an order-\(d\) chaos, using \(N\) copies in each mode gives
the divisor \((1+\log N)^d\). A weighted slice norm with \(k\)
free coordinates grows by exactly \(N^{k/2}\), while the largest
entry is unchanged. When the largest-entry term dominates, we choose
the largest replication level for which every enlarged term remains
at most \(r^dm_A\). This level is \(\lambda_A(r)\), and the resulting
logarithmic divisor gives the improvement in \cref{thm:main}.
When \(\lambda_A(r)=1\), we use no replication. The full argument
is given in \cref{sec:proof-main-results}.}

\section{Sparse Khatri--Rao embeddings}
\label{sec:applications}
\label{sec:app-khatri-rao}

The squared norm of a sparse Khatri--Rao sketch is an empirical
second moment of independent polynomial chaoses. We use the
concentration bounds above to control this average. The main result
gives an explicit row count under a condition on the support of each
tensor slice. We then estimate all normalized inner products in a
finite collection, and construct a planted relation for which the
row count is sharp up to constants depending only on the tensor order.
\appreview{For this family, the Bennett correction reduces the sufficient
projection dimension by a logarithmic factor while keeping the map
and sampling probability fixed. A matching lower bound shows that
the improved dimension is necessary for this map and its squared-norm
estimator; \cref{tab:planted-row-comparison} summarizes the comparison.}
\appreview{General embedding bounds are given in Appendix~\ref{sec:general-sketches}.}

\subsection{Construction and computational cost}
\label{sec:cost-model}

Throughout this section, take \(n_s=n\) in every mode.  We identify
\(\mathbb R^{n^d}\) with
\(\mathbb R^{n\times\cdots\times n}\), so vector Euclidean norms
are tensor Frobenius norms.  Our goal is a random map
\(\Phi\in\mathbb R^{M\times n^d}\) that approximately preserves
the norm using \(M\ll n^d\) measurements.  For a fixed pair of
tensors \(x,y\), linearity gives \(\Phi x-\Phi y=\Phi(x-y)\).
Applying a relative-error guarantee for squared norms to \(x-y\) therefore
preserves their squared distance, with the same success probability:
\[
 (1-\varepsilon)\norm{x-y}_{\mathrm F}^2
 \leq\norm{\Phi x-\Phi y}_2^2
 \leq(1+\varepsilon)\norm{x-y}_{\mathrm F}^2.
\]

A \textit{Khatri--Rao sketch} represents each row as a tensor
product of short vectors.  This representation avoids forming a
length-\(n^d\) row and, for a rank-one input, reduces each
measurement to a product of \(d\) short inner products.
The same idea is used in the tensor random projections of
Sun, Guo, Tropp, and Udell~\cite{SunGuoTroppUdell2021} and the
tensor-structured sketches of Chen and Jin~\cite{ChenJin2021}.
We make the factors sparse to further reduce storage and arithmetic
work.  Fix a common sampling probability
\(p\in(0,1]\) and a number of sketch rows \(M\geq1\).
For \(s\in[d]\), \(m\in[M]\), and \(i\in[n]\), let the variables
\(\xi^{(s)}_{mi}\) be mutually independent, with
\[
 \Pp\{\xi^{(s)}_{mi}=1\}
 =\Pp\{\xi^{(s)}_{mi}=-1\}
 =\frac p2,
 \qquad
 \Pp\{\xi^{(s)}_{mi}=0\}=1-p.
\]
Write \(\xi_m^{(s)}\coloneqq(\xi^{(s)}_{mi})_{i=1}^{n}\in\mathbb R^n\).
The sketching matrix \(\Phi\in\mathbb R^{M\times n^d}\) has rows
\[
 \Phi_{m,:}
 \coloneqq\frac{1}{\sqrt{Mp^d}}
   \bigl(\xi_m^{(1)}\otimes\cdots\otimes\xi_m^{(d)}\bigr)^{\mathsf T},
 \qquad m\in[M].
\]

On an order-\(d\) tensor \(x=(x_{\bm i})\), this map acts by
\begin{equation}\label{eq:khatri-rao-map}
 (\Phi x)_m
 \coloneqq
 \frac{1}{\sqrt{Mp^d}}
 \sum_{\bm i\in[n]^d}x_{\bm i}
 \prod_{s=1}^d\xi^{(s)}_{m i_s},
 \qquad m\in[M].
\end{equation}

Independence and centering give
\begin{equation}\label{eq:khatri-isotropy}
 \E\langle\Phi x,\Phi y\rangle=\langle x,y\rangle,
 \qquad
 \E\norm{\Phi x}_2^2=\norm{x}_{\mathrm F}^2.
\end{equation}

\appreview{We store \(\Phi\) through its \(Md\) sparse factor vectors,
recording only nonzero positions and signs. Geometric-gap sampling
avoids scanning all coordinates when \(p<1\), so generating and
storing the map costs \(O(Md(np+1))\) expected operations and storage
words. The cost of applying it depends on how the input is supplied.}

\appreview{For a rank-one tensor supplied through its factors,
\(x=u^{(1)}\otimes\cdots\otimes u^{(d)}\), each measurement is a
product of short inner products,}
\[
 (\Phi x)_m
 =\frac{1}{\sqrt{Mp^d}}
   \prod_{s=1}^d\left(\sum_{i=1}^{n}
          \xi^{(s)}_{mi}u_i^{(s)}\right).
\]
\appreview{With coordinate access to the input factors, only sampled
coordinates are needed, giving the expected cost}
\begin{equation}
 O\!\left(Md(np+1)\right)
 \quad\text{\appreview{arithmetic operations}}.
\end{equation}
\appreview{For \(np\geq1\), this is of order \(p\) times the cost with
dense factors. Reading all input factors adds \(O(dn)\) operations.}

\appreview{For a general tensor supplied through entry lookups, each row
accesses the entries selected in every mode, averaging \((np)^d\)
queries. We assume expected constant-time lookup; a lookup table can
be built from a list of nonzero entries in expected time proportional
to its length.}

\appreview{For \(N_{\rm dat}\) such tensors sharing one map, the total
expected cost of generating the map and computing the feature vectors is}
\begin{equation}
 O_d\!\left(M(np+1)+N_{\rm dat}M[1+(np)^d]\right).
\end{equation}
\appreview{The first term is the shared map-generation cost; the second
is the cost of evaluating all tensors. Storing the features takes
\(O(N_{\rm dat}M)\) words, and computing their pairwise inner products
takes \(O(N_{\rm dat}^2M)\) arithmetic operations. These counts use
unit-cost random draws, entry access, arithmetic, and output writes,
excluding bit complexity. At fixed \(p\), reducing \(M\) lowers both
evaluation costs and feature storage.}

\subsection{Tensors with bounded sampled slice counts}
\label{sec:codegree-application}
\appreview{We give a condition for the logarithmic improvement in terms
of slice support counts, the numbers of nonzero entries remaining
after some indices are fixed. These counts can be computed directly
from the list of nonzero entries.}

\appreview{For a real tensor \(A=(a_{\bm i})\), write
\(m_A\coloneqq\norm A_\infty\). Let \(D_I(A)\) denote the maximum number
of nonzero entries in any slice obtained by fixing the indices in \(I\),}
\[
 D_I(A)\coloneqq\max_{\bm i_I}
       \#\{\bm i_{I^c}:a_{\bm i}\ne0\},
 \qquad \varnothing\ne I\subsetneq[d].
\]
We require
\begin{equation}\label{eq:codegree-condition}
 p^{d-|I|}D_I(A)\leq1
 \quad(\varnothing\ne I\subsetneq[d]),\qquad
 L_A\coloneqq\frac{p^d\norm A_{\mathrm F}^2}{m_A^2}\geq1.
\end{equation}
\appreview{The first condition limits each proper slice to at most one
sampled nonzero entry on average, conditional on its fixed coordinates
being selected. The quantity \(L_A\) is the variance of an unnormalized
measurement divided by \(m_A^2\). Both conditions allow entries of
either sign.} \appreview{For example, the diagonal tensor
\(a_{i_1,\ldots,i_d}\coloneqq\ind\{i_1=\cdots=i_d\}\) has \(D_I(A)=1\),
\(m_A=1\), and \(L_A=np^d\), so the conditions hold for
\(n^{-1/d}\leq p\leq1\).}

\appreview{The slice counts can be computed by grouping the nonzero
entries according to their fixed indices. Together with
\(\norm A_{\mathrm F}\) and \(m_A\), they give the admissible range
of sampling probabilities, provided this interval is nonempty,}
\begin{equation}
 \left(\frac{m_A^2}{\norm A_{\mathrm F}^2}\right)^{1/d}
 \leq p\leq
 \min\left\{1,\min_{\varnothing\ne I\subsetneq[d]}
                D_I(A)^{-1/(d-|I|)}\right\}.
\end{equation}

For \(L\geq1\) and \(r\geq2\), define
\begin{equation}\label{eq:codegree-heavy-term}
 \Gamma_d(L,r)\coloneqq
 \frac{r^{2d}}{L\,\ell\!\left(1\vee r/L^{1/(2d-1)}\right)^{2d}}.
\end{equation}
Recall that \(\ell(u)\coloneqq1+\log u\) for \(u\geq1\).
The function \(\Gamma_d\) describes the contribution of a rare
large row to the squared-norm error.

\begin{theorem}[Sketching tensors with bounded sampled slice counts]
\label{thm:codegree-sketch}
Fix \(d\geq2\). There is a constant \(C_d\) such that the
following holds. \appreview{Let \(A\) satisfy}
\eqref{eq:codegree-condition} for the sampling probability
\(p\) in \eqref{eq:khatri-rao-map}. For every \(0<\varepsilon,\delta<1\), put
\(q_\delta\coloneqq\log(e^2/\delta)\). If \(M\) is an integer satisfying
\begin{equation}\label{eq:codegree-rows}
 M\geq C_d\left[
       \varepsilon^{-2}q_\delta+
       \varepsilon^{-1}\Gamma_d(L_A,q_\delta)\right],
\end{equation}
then
\begin{equation}\label{eq:codegree-norm-guarantee}
 \Pp\left(
   \left|\frac{\norm{\Phi A}_2^2}{\norm A_{\mathrm F}^2}-1\right|
       >\varepsilon\right)
 \leq\delta.
\end{equation}
There is no additional relation between \(L_A\), \(\delta\), and
the mode dimension.
\end{theorem}

\appreview{The two terms in \eqref{eq:codegree-rows} account for accumulated
fluctuations and rare large measurements. Keeping the input, sampling
probability, error tolerance, and failure probability fixed, the
Bennett correction divides the factor \(q_\delta^{2d}/L_A\) in the
second term of the uncorrected bound by
\(\ell(1\vee q_\delta/L_A^{1/(2d-1)})^{2d}\).
This gain matters when rare large measurements determine the row
count, as in \cref{prop:planted-relation}. When accumulated fluctuations
already dominate, the correction does not change the order of the
bound. In particular, if \(L_A\geq q_\delta^{2d-1}\), then
\(C_d\varepsilon^{-2}q_\delta\) rows suffice.}

\appreview{The theorem allows entries of either sign, and its assumptions
can be checked from the support and entry magnitudes. For tensors
with widely varying magnitudes, using the full weighted norms in
\cref{thm:main} can yield finer bounds.}

To prove \cref{thm:codegree-sketch}, we first bound the weighted
tensor norms using the slice counts and the row and column sums
of tensor flattenings. The replication estimate then controls
the moments of each normalized measurement
\(Z_m\coloneqq\sqrt M(\Phi A)_m/\norm A_{\mathrm F}\).
For an even moment order \(r\geq2\) and \(M\geq r\), applying an independent-sum
moment bound to the squared measurements gives
\[
 \left\|\frac1M\sum_{m=1}^M(Z_m^2-1)\right\|_{L_r}
 \leq C_d\left[
 \sqrt{\frac{r}{M}}+\frac{\Gamma_d(L_A,r)}M
 \right].
\]
Rounding \(\log(e^2/\delta)\) up to an even
moment order changes only \(C_d\). Under
\eqref{eq:codegree-rows}, Markov's inequality then gives failure
probability at most \(\delta\).
Details are in \cref{sec:codegree-proof}.

\subsection{Features for aligned relation data}
\label{sec:relation-features}
\appreview{We extend norm preservation for a single tensor to similarity
estimation for a finite collection, using the same random map for
all inputs. As in random-feature methods for kernel
approximation~\cite{RahimiRecht2007}, we approximate normalized
inner products by inner products of short feature vectors.
For nonzero tensors \(A,B\), define}
\[
 K(A,B)\coloneqq\frac{\langle A,B\rangle}
                  {\norm A_{\mathrm F}\norm B_{\mathrm F}},
 \qquad
 \psi(A)\coloneqq\frac{\Phi A}{\norm A_{\mathrm F}}.
\]
\appreview{For example, tensors may record the presence or strength of
multiway interactions among the same indexed
entities~\cite{KoldaBader2009}. For binary tensors, \(K(A,B)\)
measures their support overlap, normalized by the geometric mean of
their support sizes. The shared coordinate labels fix the alignment
between tensors. The following result controls all pairwise
similarities simultaneously.}

\begin{corollary}[Approximating a relation-overlap Gram matrix]
\label{cor:relation-kernel}
Let \(A_1,\ldots,A_{N_{\rm dat}}\) be \appreview{fixed tensors}
on the same coordinate set \([n]^d\),
each satisfying \eqref{eq:codegree-condition} for the same \(p\).
Fix \(0<\varepsilon,\delta<1\). If
\begin{equation}\label{eq:relation-kernel-rows}
 M\geq C_d\left[
   \varepsilon^{-2}\log\frac{e^2N_{\rm dat}^2}{\delta}+
   \varepsilon^{-1}\max_{i\in[N_{\rm dat}]}
     \Gamma_d\!\left(L_{A_i},
              \log\frac{e^2N_{\rm dat}^2}{\delta}\right)\right],
\end{equation}
then, with probability at least \(1-\delta\),
\begin{equation}\label{eq:relation-kernel-guarantee}
 \left|\langle\psi(A_i),\psi(A_j)\rangle-K(A_i,A_j)\right|
 \leq\varepsilon
 \qquad\text{for all }i,j\in[N_{\rm dat}].
\end{equation}
\end{corollary}
\appreview{The feature Gram matrix \(FF^{\mathsf T}\), with
\(F_{i,:}\coloneqq\psi(A_i)^{\mathsf T}\), is positive semidefinite
and has rank at most \(M\). It approximates the normalized-input
Gram matrix entrywise within \(\varepsilon\), and the corresponding
squared distances within \(4\varepsilon\).} \appreview{The proof bounds products of measurements directly, so the
support condition is needed only for individual tensors. See \cref{sec:codegree-proof} for details.}

\subsection{Sharp row counts for a planted relation}
\label{sec:planted-relation}
\appreview{We show that the logarithmic correction is necessary for a
family consisting of a block of ones in a sparse random background,
as in planted dense-subhypergraph models~\cite{DhawanMaoWein2025}.
The matching bounds below concern our sparse Khatri--Rao map with
its squared output norm as the estimator.}

\appreview{Fix \(d\geq2\). Let \(\tau\) be a sufficiently large positive integer
and let \(n\geq \tau^2\) be any integer. Set
$p\coloneqq\frac{\tau}{n}$ and
 $b\coloneqq\left\lfloor\frac{n}{8\tau}\right\rfloor.$
Define \(A\in\{0,1\}^{[n]^d}\) by setting
\(a_{\bm i}\coloneqq1\) for \(\bm i\in[b]^d\) and taking the remaining entries
independently \(\Ber(1/(8\tau^{d-1}))\), independently of the sketch.
The parameter \(\tau=np\) is the expected number of nonzero coordinates
per projection factor, and \(b\) is the side length of the planted block.}

\begin{proposition}[A planted relation with a sharp row count]
\label{prop:planted-relation}
\appreview{For the random tensor \(A\) and parameters defined above,
there are constants \(c>0\) and \(c_d,C_d>0\), with \(c\) absolute,}
such that, with probability at least \(1-e^{-cn/\tau}\) over \(A\),
\begin{equation}\label{eq:planted-good-event}
 L\coloneqq p^d\norm A_{\mathrm F}^2,
 \qquad \frac{\tau}{16}\leq L\leq\frac{\tau}{4},
 \qquad p^{d-|I|}D_I(A)\leq1
 \quad(\forall \varnothing\ne I\subsetneq[d]).
\end{equation}
Condition on any realization satisfying
\eqref{eq:planted-good-event}.
For every \(0<\delta<1\) with
\(\tau^{1/d}\leq\log(e^2/\delta)\leq2\tau^{1/d}\),
\[
 M\geq C_d\frac{\tau}{(\log \tau)^{2d}}
 \quad\Longrightarrow\quad
 \Pp_\Phi\!\left(
   \left|\frac{\norm{\Phi A}_2^2}{\norm A_{\mathrm F}^2}-1\right|
       >\tfrac12\right)
 \leq\delta,
\]
whereas
\[
 1\leq M\leq c_d\frac{\tau}{(\log \tau)^{2d}}
 \quad\Longrightarrow\quad
 \Pp_\Phi\!\left(
   \left|\frac{\norm{\Phi A}_2^2}{\norm A_{\mathrm F}^2}-1\right|
       >\tfrac12\right)
 >\delta.
\]
\end{proposition}

\appreview{Choosing \(\tau\) of order \((\log n)^d\) gives the following
matching row bounds at polynomially small failure probabilities.}

\begin{corollary}[Optimal row count at polynomial confidence]
\label{cor:planted-polynomial-confidence}
Fix \(d\geq2\) and \(a>0\). For sufficiently large \(n\), set
\(\tau\coloneqq\lfloor[\log(e^2n^a)]^d\rfloor\), and draw \(A\)
\appreview{using the construction above.} On the event
\eqref{eq:planted-good-event}, the least row count $M$ for which
\[
 \Pp_\Phi\!\left(
 \left|\frac{\norm{\Phi A}_2^2}{\norm A_{\mathrm F}^2}-1\right|
       >\tfrac12\right)\leq n^{-a}
\]
is of order
\begin{equation}\label{eq:planted-polynomial-rows}
 \frac{(\log n)^d}{(\log\log n)^{2d}},
\end{equation}
with constants depending only on \(d,a\).
Here \(p=\tau/n\asymp_{d,a}(\log n)^d/n\), and the event
has probability at least
\(1-\exp\{-cn/\tau\}\) over \(A\), where \(c>0\) is absolute.
\end{corollary}

\appreview{For the planted family, our upper bound gives a sufficient
row count, while the matching lower bound shows that fewer rows
cannot achieve the stated guarantee, up to constants. We now compare
with earlier sufficient bounds, keeping the tensors, sampling probability,
error tolerance, and failure probability fixed. Under the same slice-count condition,
the bounds of Dai and Wang~\cite{DaiWang2026} or Schudy and
Sviridenko~\cite[Theorem~1.3]{schudy2011bernstein} give the sufficient
row count below, with \(L=L_A\),}
\begin{equation}\label{eq:relation-uncorrected-rows}
 M\geq C_d\left\{\varepsilon^{-2}\log\frac{e^2}{\delta}+
        \frac{\varepsilon^{-1}[\log(e^2/\delta)]^{2d}}{L}\right\}.
\end{equation}
\appreview{The averaging argument is unchanged, so the gain comes from our sharper bound on each
measurement. For the planted family, \(L\asymp\tau\) and
\(\log(e^2/\delta)\asymp\tau^{1/d}\); at relative error \(1/2\),
\cref{tab:planted-row-comparison} shows that this gain closes the gap
to the lower bound.}

\begin{table}[htbp]
\centering
\caption{\appreview{Sufficient and necessary row counts for the same sparse
Khatri--Rao map on the planted family.}}
\label{tab:planted-row-comparison}
\begin{tabular}{lc}
\toprule
Bound for the same sparse map & Number of rows\\
\midrule
\appreview{Our sufficient bound}
  & \(C_d \tau/(\log \tau)^{2d}\) suffices\\
\appreview{Our necessary lower bound}
  & \(c_d \tau/(\log \tau)^{2d}\) is necessary\\
\appreview{Earlier uncorrected sufficient bound}
  & \(C_d \tau\) suffices\\
\bottomrule
\end{tabular}
\end{table}

\appreview{With \(\tau\asymp(\log n)^d\), the sufficient row count drops
from order \((\log n)^d\) to the sharp order in
\eqref{eq:planted-polynomial-rows}. The matching lower bound establishes
optimality up to constants for this family, this map, and the squared
output-norm estimator at the stated accuracy and confidence.
For other tensors, the two upper
bounds remain sufficient but need not be sharp.}

\appreview{For context, distribution-only estimates give sufficient
row counts for arbitrary fixed tensors, without the slice-count
condition. Writing \(\theta_p^2\coloneqq[p\log(1+1/p)]^{-1}\), a standard
sub-Gaussian estimate for each measurement followed by
\cref{lem:moment-averaging} gives the sufficient condition}
\[
 M\geq C_d\left\{\varepsilon^{-2}\theta_p^{4d}\log\frac{e^2}{\delta}+
 \varepsilon^{-1}\theta_p^{2d}[\log(e^2/\delta)]^d\right\}.
\]
\appreview{For \(d=2\), Chen and
Jin~\cite[Theorem~2.1 and equation~(46)]{ChenJin2021} give the sufficient condition}
\[
 M\geq C\left\{\varepsilon^{-2}p^{-4}\log\frac{e^2}{\delta}+
 \varepsilon^{-1}p^{-2}[\log(e^2/\delta)]^2\right\}
\]
\appreview{for arbitrary fixed inputs. These larger sufficient row counts
in the planted regime illustrate the benefit of using the tensor's
support structure. The comparisons are derived in
\cref{sec:relation-baselines}.}

\appreview{The shared-feature setting follows similarly by averaging
products of measurements and taking a union bound over tensor pairs.
In the uncorrected comparison, use \(L=\min_iL_{A_i}\) and replace
\(\delta\) by \(\delta/N_{\rm dat}^2\). For a polynomial-size collection
satisfying \eqref{eq:codegree-condition} with \(L_{A_i}\asymp(\log n)^d\),
\cref{cor:relation-kernel} gives the same row-count order at fixed error
and polynomially small failure probability. A matching lower bound
follows from the diagonal Gram entries if the collection contains a
tensor from the planted family.}

\subsection{General embeddings}
\label{sec:additional-embeddings}
\appreview{The method also applies to tensors that do not satisfy the
slice-count condition. In Appendix~\ref{sec:general-sketches}, we give bounds
for preserving a single tensor's norm, pairwise distances in a finite
collection, or all norms in a fixed subspace. These bounds use more
detailed information about the tensor coefficients. When the projection
factors are dense random signs, they recover the standard row-count
guarantee.}

\section{Proofs of the main results}
\label{sec:proof-main-results}

\subsection{The replication estimate}
\label{sec:technical-replication}

Recall the weighted norms \(\cS_{I,\cJ}(A)\), the moment scales
\(\Delta_I(A;r)\) from \eqref{eq:Delta}, and
\(\ell(N)\coloneqq1+\log N\).

\begin{proposition}[Replication bound]
\label{prop:replication-profile}
Let \(A\) be a real \(d\)-tensor and let the mutually independent
inputs satisfy \eqref{eq:assump-X}.  For every \(r\geq2\) and
\(N\geq1\),
\begin{equation}\label{eq:fixed-replication-real-moment}
 \norm{Z_A(X)}_{L_r}
 \lesssim_d
 \frac{\displaystyle\max_{I\subseteq[d]}
          N^{|I^c|/2}\Delta_I(A;r)}
      {\ell(N)^d}.
\end{equation}
Consequently, for every \(t>0\),
\begin{equation}\label{eq:replication-mixed-tail}
 \Pp(\abs{Z_A(X)}\geq t)
 \leq2\exp\left[
 -c_d\min_{\substack{I\subseteq[d]\\\cJ\in\Pi(I^c)}}
 \left(\frac{t\ell(N)^d}
 {N^{|I^c|/2}\cS_{I,\cJ}(A)}\right)^{1/\kappa_{I,\cJ}}
 \right],
 \qquad \kappa_{I,\cJ}=\abs I+\abs{\cJ}/2.
\end{equation}
\end{proposition}

The parameter \(N\) balances two effects.  Replication enlarges a
norm with \(k\) free coordinates by \(N^{k/2}\), but the moment
comparison gains the divisor \(\ell(N)^d\).  In particular, the
largest-entry term, for which \(k=0\), is not enlarged.
We first prove the estimate for integer \(N\), then extend it by
rounding; \cref{sec:explicit-proofs} chooses \(N\) to obtain the
main results.

\subsection{Proof of the replication estimate}
\label{sec:proof-replication}

The argument combines a coordinatewise moment comparison, a known
bound for sparse-exponential chaoses, and an exact calculation of
how replication changes the weighted norms.

\subsubsection{A scalar comparison and its multilinear consequence}
\label{sec:replicated-moment-comparison}

Let \(\sigma\) be a Rademacher sign and \(E\sim\operatorname{Exp}(1)\),
independent of each other.  The symmetric exponential variable
\(\zeta\coloneqq\sigma E\) satisfies
\[
 \E\zeta^m=0\quad(m\ \text{odd}),\qquad
 \E\zeta^m=m!\quad(m\ \text{even}),\qquad
 \E\exp(\abs\zeta/2)=2.
\]
For an integer \(N\geq1\), replace each mode-\(s\) coordinate \(i_s\)
by \((i_s,u_s)\in[n_s]\times[N]\), and define
\[
 a^{[N]}_{(i_1,u_1),\ldots,(i_d,u_d)}\coloneqq a_{i_1,\ldots,i_d},
 \qquad
 Y_{i,u}^{(s)}\coloneqq\delta_{i,u}^{(s)}\sigma_{i,u}^{(s)}E_{i,u}^{(s)},
 \qquad
 \delta_{i,u}^{(s)}\sim\Ber(p_i^{(s)}).
\]
All Bernoulli variables, signs, and exponential variables are
mutually independent.  The replicated tensor \(A^{[N]}\) has
dimensions \(n_1N\times\cdots\times n_dN\); each replica inherits
the original probability \(p_i^{(s)}\).

\begin{lemma}[Replicated moment comparison]
\label{lem:replicated-moment-comparison}
Under the assumptions of \cref{prop:replication-profile}, for every
even integer \(r\geq2\) and integer \(N\geq1\),
\[
 \norm{Z_A(X)}_{L_r}
 \leq\frac{(2\sqrt e)^d}{\ell(N)^d}
       \norm{Z_{A^{[N]}}(Y)}_{L_r}.
\]
\end{lemma}

\begin{proof}
\emph{Step 1: symmetrize the bounded inputs.}
Let \(X'\) be an independent copy of the entire input family and
put \(V_i^{(s)}=(X_i^{(s)}-X_i^{\prime(s)})/2\).
These coordinates are independent and symmetric, and
\[
 |V_i^{(s)}|\leq1,\qquad
 \E(V_i^{(s)})^2
 =\tfrac12\E(X_i^{(s)})^2\leq p_i^{(s)}.
\]
For each monomial, independence and centering give
\[
 \E_{X'}\prod_{s=1}^d
 (X_{i_s}^{(s)}-X_{i_s}^{\prime(s)})
 =\prod_{s=1}^d X_{i_s}^{(s)}.
\]
Thus \(Z_A(X)=\E_{X'}Z_A(X-X')\).  Conditional Jensen's
inequality, followed by degree-\(d\) homogeneity, yields
\[
 \E|Z_A(X)|^r\leq\E|Z_A(X-X')|^r
 =2^{dr}\E|Z_A(V)|^r.
\]

\emph{Step 2: compare the scalar moments.}
Put \(S_i^{(s)}=\sum_{u=1}^N Y_{i,u}^{(s)}\).
For every even integer \(m\geq2\), independence gives the
multinomial expansion
\[
 \E(S_i^{(s)})^m
 =\sum_{\substack{m_1+\cdots+m_N=m\\m_u\geq0}}
   \frac{m!}{m_1!\cdots m_N!}
   \prod_{u=1}^N\E(Y_{i,u}^{(s)})^{m_u}.
\]
Terms with an odd \(m_u\) vanish, and all remaining terms are
nonnegative.  Keeping the \(N\) terms with one exponent equal to
\(m\) and all others zero gives
\[
 \E(S_i^{(s)})^m\geq Np_i^{(s)}m!
 \geq \frac{p_i^{(s)}\ell(N)^m}{e}.
\]
The second inequality follows from the exponential series:
\(\exp(\ell(N))=eN\geq\ell(N)^m/m!\).
On the other hand, boundedness gives
\(\E|V_i^{(s)}|^m\leq\E(V_i^{(s)})^2\leq p_i^{(s)}\).
Since \(e^{m/2-1}\geq1\), we obtain
\begin{equation}\label{eq:scalar-replication-comparison}
 \E(V_i^{(s)})^m
 \leq
 \E\left(\frac{\sqrt e\,S_i^{(s)}}{\ell(N)}\right)^m
 \qquad(m\geq2\ \text{even}).
\end{equation}
Both sides have vanishing odd moments.

\emph{Step 3: replace the coordinates one at a time.}
Take the comparison variables independent of \(V\).
After conditioning on every coordinate except one, the chaos is
\(a+bv\), where \(a,b\) are fixed real numbers under the
conditioning.  For symmetric \(v\) and even \(r\),
\[
 \E_v|a+bv|^r
 =\sum_{\substack{0\leq m\leq r\\m\ {\rm even}}}
   \binom rm a^{r-m}b^m\E v^m.
\]
Both \(m\) and \(r-m\) are even, so every coefficient is
nonnegative.
Consequently, \eqref{eq:scalar-replication-comparison} allows
\(v=V_i^{(s)}\) to be replaced by
\(\sqrt e\,S_i^{(s)}/\ell(N)\) without decreasing the conditional
\(r\)-th moment.  Integrate over the other coordinates and repeat
for the entire finite input family.  Every monomial contains
exactly \(d\) coordinates, hence
\[
 \norm{Z_A(V)}_{L_r}
 \leq\frac{e^{d/2}}{\ell(N)^d}\norm{Z_A(S)}_{L_r}.
\]

\emph{Step 4: identify the replicated chaos.}
Distributing the sums gives the exact identity
\[
 Z_A(S)
 =\sum_{\bm i}a_{\bm i}
       \prod_{s=1}^d\left(\sum_{u_s=1}^N Y_{i_s,u_s}^{(s)}\right)
 =\sum_{\bm i}\sum_{\bm u\in[N]^d}
       a_{\bm i}\prod_{s=1}^dY_{i_s,u_s}^{(s)}
 =Z_{A^{[N]}}(Y).
\]
Combining this with Steps 1 and 3 proves the lemma.
\end{proof}

\subsubsection{The reference bound and the cost of replication}
\label{sec:sparse-exponential-input}

The following specialization of
Dai and Wang~\cite[Theorem~1]{DaiWang2026} estimates the comparison
chaos.

\begin{lemma}[Sparse-exponential chaos]
\label{lem:sparse-exponential-chaos}
Let \(W_i^{(s)}\coloneqq\delta_i^{(s)}\chi_i^{(s)}\), where all Bernoulli variables and
amplitudes are mutually independent,
\(\delta_i^{(s)}\sim\Ber(p_i^{(s)})\),
\(\E\chi_i^{(s)}=0\), and
\(\E\exp(\abs{\chi_i^{(s)}}/K_{\mathrm{exp}})\leq2\).
For every real \(d\)-tensor \(A\) and every \(r\geq2\),
\begin{equation}
 \norm{Z_A(W)}_{L_r}
 \leq C_dK_{\mathrm{exp}}^d
       \max_{I\subseteq[d]}\Delta_I(A;r).
\end{equation}
\end{lemma}

\appreview{To obtain this specialization, regard each mode-coordinate
pair as a distinct variable. Centering makes all expected derivatives
of order below \(d\) vanish; the top derivative consists of permuted
blocks of \(A\), whose weighted slice norms give \(\Delta_I(A;r)\).
The number of blocks and partitions depends only on \(d\), so their
sum is bounded by a constant times the displayed maximum.}

\begin{lemma}[Exact norm scaling]
\label{lem:norm-scaling}
For integer \(N\geq1\), evaluate the norms of \(A^{[N]}\)
using the inherited probabilities
\(p_{(i,u)}^{[N],(s)}\coloneqq p_i^{(s)}\).  Then
\[
 \cS_{I,\cJ}(A^{[N]})
 =N^{|I^c|/2}\cS_{I,\cJ}(A),\qquad
 \Delta_I(A^{[N]};r)
 =N^{|I^c|/2}\Delta_I(A;r).
\]
\end{lemma}

\begin{proof}
Fix \(I\), the values of \(\bm i_I\), and
\(\cJ\in\Pi(I^c)\).  Write \(U=I^c\) and consider the weighted
slice
\[
 b_{\bm i_U}=a_{\bm i}
             \prod_{s\in U}\sqrt{p_{i_s}^{(s)}}.
\]
The values of the fixed replica indices \(\bm u_I\) do not
affect this slice.  If \(U=\emptyset\), its norm is the absolute
value of one entry, which is unchanged.  Assume \(U\neq\emptyset\).

For each block \(J\in\cJ\), take a test array
\(x^{(J)}_{\bm i_J,\bm u_J}\) with Euclidean norm at most one,
and set
\[
 y^{(J)}_{\bm i_J}
 =N^{-|J|/2}\sum_{\bm u_J\in[N]^J}
                   x^{(J)}_{\bm i_J,\bm u_J}.
\]
Cauchy--Schwarz, applied to the \(N^{|J|}\) replica indices,
gives
\[
 \sum_{\bm i_J}|y^{(J)}_{\bm i_J}|^2
 \leq\sum_{\bm i_J,\bm u_J}
                 |x^{(J)}_{\bm i_J,\bm u_J}|^2\leq1.
\]
Because the blocks partition \(U\), direct expansion yields
\[
 \sum_{\bm i_U,\bm u_U}b_{\bm i_U}
       \prod_{J\in\cJ}x^{(J)}_{\bm i_J,\bm u_J}
 =N^{|U|/2}\sum_{\bm i_U}b_{\bm i_U}
       \prod_{J\in\cJ}y^{(J)}_{\bm i_J}.
\]
Taking the supremum over admissible \(x^{(J)}\) proves that the
replicated partition norm is at most
\(N^{|U|/2}\norm b_{\cJ}\).
Conversely, for any admissible \(y^{(J)}\), take
\[
 x^{(J)}_{\bm i_J,\bm u_J}
       =N^{-|J|/2}y^{(J)}_{\bm i_J}.
\]
These arrays have the same Euclidean norms as the corresponding
\(y^{(J)}\), and the preceding identity gives the reverse
inequality.  Hence the norm scales exactly by \(N^{|U|/2}\).
Maximizing over the fixed original and replica indices proves
the assertion for \(\cS_{I,\cJ}\); maximizing over partitions
after multiplying by \(r^{\kappa_{I,\cJ}}\) proves the assertion
for \(\Delta_I\).
\end{proof}

\subsubsection{Completion of the replication estimate}
\label{sec:proof-completion}

\begin{proof}[Proof of \cref{prop:replication-profile}]
For even integers \(r\geq2\) and integers \(N\geq1\),
apply the preceding three lemmas with \(K_{\mathrm{exp}}=2\) to obtain
\[
 \norm{Z_A(X)}_{L_r}
 \lesssim_d
 \frac{\norm{Z_{A^{[N]}}(Y)}_{L_r}}{\ell(N)^d}
 \lesssim_d
 \frac{\max_{I\subseteq[d]}\Delta_I(A^{[N]};r)}{\ell(N)^d}
 =
 \frac{\max_{I\subseteq[d]}N^{|I^c|/2}\Delta_I(A;r)}
      {\ell(N)^d}.
\]
The first two constants depend only on \(d\).

For real \(r\geq2\) and \(N\geq1\), set
\(\bar r=2\lceil r/2\rceil\) and \(\bar N=\lceil N\rceil\).
Since \(\bar r\leq2r\), \(\bar N\leq2N\), and
\(\kappa_{I,\cJ}\leq d\), rounding increases the numerator
by at most \(2^{3d/2}\) and cannot decrease \(\ell(N)^d\).
Monotonicity of \(L_r\)-norms therefore proves
\eqref{eq:fixed-replication-real-moment} for real \(r,N\).

Finally, since \(Z_A(X)\) is centered, the standard moment-to-tail
implication \cite[Proposition~3.3]{gotze2021concentration}
applied to \eqref{eq:fixed-replication-real-moment} gives
\eqref{eq:replication-mixed-tail}.  
\end{proof}

\subsection{Choosing the replication level}
\label{sec:explicit-proofs}

\begin{proof}[Proof of \cref{thm:main}]
In \cref{prop:replication-profile}, group the terms by the number
\(k=|I^c|\) of free coordinates.  The case \(k=0\) is the
largest-entry term \(\Delta_{[d]}(A;r)=r^dm_A\), and the maximum
over the terms with \(k\geq1\) is \(B_k(A;r)\).  Thus
\begin{equation}\label{eq:common-replication-moment}
 \norm{Z_A(X)}_{L_r}
 \lesssim_d\frac{
 \max\{r^dm_A,\max_{1\leq k\leq d}N^{k/2}B_k(A;r)\}}
 {\ell(N)^d}.
\end{equation}
To keep the numerator equal to \(r^dm_A\), it suffices that
\[
 N^{k/2}B_k(A;r)\leq r^dm_A
 \quad\text{for every }k,\qquad\text{or equivalently}\qquad
 N\leq\min_{1\leq k\leq d}
       \left(\frac{r^dm_A}{B_k(A;r)}\right)^{2/k}.
\]
If \(\lambda_A(r)>1\), its definition makes
\(N=\lambda_A(r)\) the largest admissible common replication
level.  Substitution into \eqref{eq:common-replication-moment}
proves \eqref{eq:effective-main-moment} and hence also
\eqref{eq:explicit-main-moment}.
If \(\lambda_A(r)=1\), use \(N=1\) instead.  Since \(\ell(1)=1\),
\[
 \max\{r^dm_A,\max_kB_k(A;r)\}
 \leq r^dm_A+\max_kB_k(A;r),
\]
which is precisely the right-hand side of
\eqref{eq:explicit-main-moment} in this case.
\end{proof}

\begin{proof}[Proof of \cref{cor:main-tail}]
Put \(q_0=(t/m_A)^{1/d}\).
If \(\eta_A(t)=1\), take \(N=1\) in
\eqref{eq:replication-mixed-tail}.
The entry term gives the rate \(q_0\); every other term gives
\((t/\cS_{I,\cJ}(A))^{1/\kappa_{I,\cJ}}\).
Since \(\ell(\eta_A(t))=1\), this is exactly the claimed
minimum in \eqref{eq:explicit-mixed-tail}.

Now suppose \(\eta_A(t)>1\), and take
\(N=\eta_A(t)\), \(\Lambda=\ell(N)\).
For \(I\subsetneq[d]\), write \(k=|I^c|\geq1\) and
\(\kappa=\kappa_{I,\cJ}\).  The definition of \(\eta_A(t)\)
implies
\[
 N\leq
 \left(\frac{m_Aq_0^{d-\kappa}}{\cS_{I,\cJ}(A)}\right)^{2/k},
 \qquad\text{hence}\qquad
 N^{k/2}\cS_{I,\cJ}(A)q_0^\kappa\leq t.
\]
Since \(\Lambda\geq1\) and \(\kappa\leq d\),
\[
 \frac{N^{k/2}\cS_{I,\cJ}(A)(q_0\Lambda)^\kappa}{\ell(N)^d}
 \leq t\Lambda^{\kappa-d}\leq t.
\]
The entry term satisfies equality:
\(m_A(q_0\Lambda)^d/\ell(N)^d=t\).
Consequently, every rate in
\eqref{eq:replication-mixed-tail} is at least \(q_0\Lambda\).
This proves the stronger estimate
\(\Pp(|Z_A(X)|\geq t)\leq2e^{-c_dq_0\Lambda}\), which implies
\eqref{eq:explicit-mixed-tail} because its displayed minimum
is at most \(q_0\Lambda\).
\end{proof}

\begin{proof}[Proof of \cref{cor:gamma-form}]
Write \(\gamma=\gamma_d(A)\).

\emph{Step 1: the replication level.}
Fix \(I\subsetneq[d]\) and \(\cJ\in\Pi(I^c)\), and put \(k=\abs{I^c}\geq1\) and
\(\kappa=\kappa_{I,\cJ}\).  By the definition of \(\gamma\),
\(\cS_{I,\cJ}(A)\leq m_A\gamma^{d-\kappa}\).
For \(u=r/\gamma\geq1\), the candidate of this pair in the minimum
defining \(\lambda_A(r)\) satisfies
\[
 \left(\frac{r^dm_A}{r^\kappa\cS_{I,\cJ}(A)}\right)^{2/k}
 \geq u^{\,2(d-\kappa)/k}\geq u.
\]
The last inequality uses \(d-\kappa=k-\abs\cJ/2\geq k/2\), because
\(\abs\cJ\leq k\), and \(u\geq1\).  Taking the minimum over all pairs gives
\(\lambda_A(r)\geq r/\gamma\) for \(r\geq\gamma\).
If \(r\geq\max\{2,e\gamma\}\), then \(\lambda_A(r)\geq e>1\) and
\(\ell(\lambda_A(r))\geq1+\log(r/\gamma)\).  Substituting this into
\eqref{eq:effective-main-moment} proves \eqref{eq:gamma-moment}.

\emph{Step 2: the tail.}
Put \(q_0=(t/m_A)^{1/d}\).  For \(I,\cJ\) as above, the candidate of this pair
in the definition of \(\eta_A(t)\) equals
\((m_Aq_0^{d-\kappa}/\cS_{I,\cJ}(A))^{2/k}\).
If \(q_0\geq\gamma\), the computation of Step 1, with \(r\) replaced by
\(q_0\), bounds it below by \(q_0/\gamma\).  Hence, for every \(t>0\),
\[
 \eta_A(t)\geq1\vee\frac{(t/m_A)^{1/d}}{\gamma}.
\]
This is an algebraic comparison and does not require \(q_0\geq2\).
Since \(\ell\) is increasing, \eqref{eq:explicit-mixed-tail} implies
\eqref{eq:gamma-tail}.
\end{proof}

For the quadratic corollary we also need forward decoupling
for diagonal-free polynomials
\cite[Theorem~1]{deLaPenaMontgomerySmith1995}.
If \(A\) is diagonal-free and
\(Q_A(X)\coloneqq\sum_{\bm i}a_{\bm i}X_{i_1}\cdots X_{i_d}\), then
\begin{equation}\label{eq:decoupling-comparison}
\begin{split}
 \Pp(\abs{Q_A(X)}\geq t)
 &\leq C_d\Pp\left(
  \abs{Z_A(X^{(1)},\ldots,X^{(d)})}\geq t/C_d\right),\\
 \norm{Q_A(X)}_{L_r}
 &\lesssim_d\norm{Z_A(X^{(1)},\ldots,X^{(d)})}_{L_r},
 \qquad r\geq2,
\end{split}
\end{equation}
where the input vectors on the right are independent copies.
Only the forward inequality is used; it does not require symmetry
of \(A\).  The moment comparison follows by integrating the
cited tail inequality, with a constant depending only on \(d\).

\begin{proof}[Proof of \cref{cor:quadratic-moment,cor:rectangular-cubic}]
\emph{Step 1: reduction to the decoupled case.}
The cubic chaos is already decoupled.  For \(d=2\), the zero diagonal allows us
to apply the moment comparison in
\eqref{eq:decoupling-comparison} to \(Q_A(X)=X^{\mathsf T}AX\).
Combining it with \eqref{eq:fixed-replication-real-moment}
gives the same common-\(N\) moment estimate for the quadratic
form, with a different universal constant.  The same
moment-to-tail implication and choices of \(N\) therefore give the
conclusions of \cref{thm:main,cor:main-tail,cor:gamma-form} for \(Q_A\),
with \(p_i^{(s)}=p_i\).
In particular, decoupling is applied before the moment-to-tail
conversion, so its constants do not change the arguments of
\(\lambda_A\), \(\eta_A\), or \(\gamma_d\).

\emph{Step 2: identify all non-entry norms.}
For \(d=2\), fixing one coordinate leaves a weighted row or
column.  Symmetry of \(A\) makes their maximal Euclidean norms
equal to \(\norm{AD_p}_{2,\infty}\).
With neither coordinate fixed, the one-block partition gives
the Frobenius norm of \(D_pAD_p\), and the two-block partition
gives its operator norm.  Since
\(\kappa=2-k+|\cJ|/2\), this yields
\[
 B_1(A;r)=r^{3/2}\norm{AD_p}_{2,\infty},\qquad
 B_2(A;r)=\max\{\sqrt r\norm{D_pAD_p}_{\mathrm F},
                     r\norm{D_pAD_p}_{\mathrm{op}}\}.
\]

For \(d=3\), a slice with \(k\) free coordinates has weight
\(p^{k/2}\), and \(\kappa=3-k+|\cJ|/2\).
With one free coordinate, the slices are vectors, so the only
partition gives \(V_1\) with exponent \(\kappa=5/2\).
With two free coordinates, the slices are matrices: the one-block
partition gives their weighted Frobenius norm \(V_2\), with
\(\kappa=3/2\), while the two-block partition gives their weighted
operator norm \(W_2\), with \(\kappa=2\).

With all three coordinates free, one block gives the weighted
Frobenius norm \(V_3\), with \(\kappa=1/2\); two blocks give the
weighted matrix-flattening norms, whose maximum is \(W_3\), with
\(\kappa=1\); and three blocks give the weighted injective norm
\(U_3\), with \(\kappa=3/2\).
These scales already include the maxima over the applicable
fixed coordinates and choices of modes.  Therefore
\[
\begin{aligned}
 B_1(A;r)&=r^{5/2}V_1,\\
 B_2(A;r)&=\max\{r^{3/2}V_2,r^2W_2\},\\
 B_3(A;r)&=\max\{\sqrt r\,V_3,rW_3,r^{3/2}U_3\}.
\end{aligned}
\]

\rev{\emph{Step 3: apply the bounds with a fixed threshold.}}
\rev{By \eqref{eq:gamma-threshold}, a non-entry scale} \(\cS\) with exponent \(\kappa\)
contributes \((\cS/m_A)^{1/(d-\kappa)}\) to \(\gamma_d(A)\).
For \(d=2\), the scales \(\norm{AD_p}_{2,\infty}\),
\(\norm{D_pAD_p}_{\mathrm F}\), \(\norm{D_pAD_p}_{\mathrm{op}}\) have
\(\kappa=3/2,\,1/2,\,1\), which gives the three terms of \(\gamma_2(A)\).
For \(d=3\), the scales \(V_1,V_2,W_2,V_3,W_3,U_3\) have
\(\kappa=5/2,\,3/2,\,2,\,1/2,\,1,\,3/2\), which gives the six terms in
\eqref{eq:cubic-threshold}.  Thus \(\gamma_2(A)\) and \(\gamma_3(A)\) are the
\rev{thresholds in \eqref{eq:gamma-threshold}, and} \eqref{eq:gamma-moment} gives
\eqref{eq:quadratic-moment} and \eqref{eq:rectangular-cubic-moment}.
Finally, the non-entry rates in \eqref{eq:gamma-tail} are
\((t/\cS)^{1/\kappa}\).  For \(d=2\), their exponents are
\(2/3,\,2,\,1\), respectively; for \(d=3\), they are
\(2/5,\,2/3,\,1/2,\,2,\,1,\,2/3\).  Inserting these rates into
\eqref{eq:gamma-tail} proves \eqref{eq:quadratic-tail} and
\eqref{eq:rectangular-cubic-tail}.
\end{proof}

\section{Proofs for sparse Khatri--Rao embeddings}
\label{sec:application-proofs}
\label{sec:proof-khatri-rao}

\label{sec:codegree-proof}
We first derive a moment bound for one measurement from the
slice-count assumptions. An elementary independent-sum estimate then
converts the moments of the squared measurements into the stated
row count. The same argument handles products of measurements and
proves the Gram-matrix guarantee. We next prove the planted example
and the comparison bounds.

\subsection{Partition norms and the row moment curve}
Fix \(A\) as in \cref{thm:codegree-sketch}, and abbreviate
\(L=L_A\). Replacing \(A\) by \(A/m_A\), we may assume
\(\norm A_\infty=1\) and \(p^d\norm A_{\mathrm F}^2=L\).
For every proper slice with \(k\) free coordinates, its weighted
Frobenius norm is at most
\[
 \left(p^kD_I(A)\right)^{1/2}\leq1.
\]
This also bounds every partition norm of that slice. For a
nontrivial full-tensor partition, first split its blocks into
two nonempty sets of modes \(U,V\). The flattening of
\(p^{d/2}|A|\) across \(U,V\) has maximum row and column
sums bounded, respectively, by
\[
 p^{d/2}D_U(A)\leq p^{d/2-|V|},\qquad
 p^{d/2}D_V(A)\leq p^{d/2-|U|}.
\]
\appreview{Schur's test therefore bounds the operator norm by one.
Taking absolute values and grouping the partition's test arrays into
tensor products gives the same bound for every nontrivial partition
norm of \(p^{d/2}A\). Consequently,}
\begin{equation}
 \begin{aligned}
 \cS_{\varnothing,\{[d]\}}(A)&=\sqrt L,\qquad
 \cS_{\varnothing,\cJ}(A)\leq1\quad(|\cJ|\geq2),\\
 B_k(A;r)&\leq r^{d-k/2}\quad(1\leq k<d),\\
 B_d(A;r)&\leq\max\{\sqrt{rL},r^{d/2}\},
 \qquad B_0(A;r)=r^d.
 \end{aligned}
\end{equation}

\appreview{For one row, write \(\xi_i^{(s)}=\xi_{1i}^{(s)}\) and
\(S_A=\sum_{\bm i}a_{\bm i}\prod_s\xi_{i_s}^{(s)}\).
Set \(r_0=L^{1/(2d-1)}\). When \(r>r_0\), we increase replication
while keeping the Frobenius and slice terms below \(r^d\), thereby
gaining the logarithmic divisor. This leads to the following choice
in \cref{prop:replication-profile}, with \(N=1\) when \(r\leq r_0\),}
\begin{equation}
 N=\min\left\{r,\max\left\{1,
                  \left(\frac{r^{2d-1}}L\right)^{1/d}
                         \right\}\right\}.
\end{equation}
If \(L\geq r^{2d-1}\), then \(N=1\), and every
term in the numerator of the replication bound is at most
\(\sqrt{rL}\). \appreview{Otherwise, the choice of \(N\) ensures
\(N^{d/2}\sqrt{rL}\leq r^d\), while \(N\leq r\) bounds the
remaining terms by}
\[
 N^{k/2}r^{d-k/2}\leq r^d,
 \qquad
 N^{d/2}r^{d/2}\leq r^d.
\]
Since \(r_0\geq1\) and \((2d-1)/d\geq1\),
\[
 \log N
 =\min\left\{\log r,
       \frac{2d-1}{d}\log(1\vee r/r_0)\right\}
 \geq\log(1\vee r/r_0).
\]
Both cases therefore give, for every \(r\geq2\),
\begin{equation}\label{eq:codegree-row-moment}
 \norm Z_{L_r}\leq C_d\left[
    \sqrt r+
    \frac{r^d}{\sqrt L\ell(1\vee r/r_0)^d}\right],
 \qquad Z=\frac{S_A}{\sqrt L}.
\end{equation}
Here \(\E Z^2=1\), by independence and centering. For
\(Y=Z^2-1\), applying \eqref{eq:codegree-row-moment} at
\(r=2q\), and using
\(\ell(1\vee2q/r_0)\geq\ell(1\vee q/r_0)\), yields
\begin{equation}\label{eq:codegree-square-moment}
 \norm Y_{L_q}\leq C_d q(1+\rho_L(q)),\qquad
 \rho_L(q)=\frac{q^{2d-1}}
 {L\ell(1\vee q/r_0)^{2d}},\qquad q\geq2.
\end{equation}
Restoring the original scale of \(A\), the same \(Z\)
is \(\sqrt M(\Phi A)_1/\norm A_{\mathrm F}\).

\subsection{\appreview{Averaging the measurements}}
We use the following even-moment estimate to average the squared
rows. See Lata{\l}a~\cite{Latala97} for general moment estimates for independent sums.

\begin{lemma}[Moments of an independent sum]\label{lem:moment-averaging}
Let \(Y_1,\ldots,Y_M\) be independent identically distributed
centered real random variables, and let \(r\geq2\) be even.
If \(\E|Y_1|^r<\infty\), put
\[
 H=\max_{q\in\{2,\ldots,r\}}
       \frac{r}{q}\left(\frac{M}{r}\right)^{1/q}\norm{Y_1}_{L_q}.
\]
Then \(\norm{\sum_{m=1}^MY_m}_{L_r}\leq2e^2H\).
\end{lemma}
\begin{proof}
The case \(H=0\) is immediate. Otherwise,
\(\E|Y_1|^q\leq H^q(q/r)^q\,r/M\).
Since the variables are centered, terms with an index of multiplicity
one vanish in the multinomial expansion. Grouping the remaining
terms by the number \(h\) of distinct indices and taking absolute
values gives
\[
 \E\left|\sum_mY_m\right|^r
 \leq H^r\frac{r!}{r^r}
 \sum_{h=1}^{\min(M,r/2)}\binom Mh\left(\frac{r}{M}\right)^h
 \sum_{\substack{q_1+\cdots+q_h=r\\q_j\geq2}}
           \prod_{j=1}^h\frac{q_j^{q_j}}{q_j!}.
\]
The expansion is valid because \(r\) is even. Using \(r!\leq r^r\),
\(q!\geq(q/e)^q\), and
\(\binom Mh(r/M)^h\leq r^h/h!\), together with the bound
\(2^r\) on the number of compositions for each \(h\), bounds
the right-hand side by
\(H^r e^r2^r\sum_{h\geq0}r^h/h!=(2e^2H)^r\).
Taking the \(r\)-th root proves the claim.
\end{proof}

\appreview{We now apply the lemma to the centered squared measurements
\(Y=Z^2-1\).} Let \(r\geq2\) be even and \(M\geq r\), and set
\(R=M/r\geq1\). For independent copies of \(Y\), \cref{lem:moment-averaging} gives
\[
 \left\|\sum_{m=1}^M Y_m\right\|_{L_r}\leq C_dr\max_{2\leq q\leq r}R^{1/q}(1+\rho_L(q)).
\]
The maximum may be taken over real \(q\) for this upper bound.
For \(q\leq r_0\), we have \(\rho_L(q)\leq1\).
For \(q\geq r_0\), the logarithmic derivative of
\(\rho_L(q)R^{1/q}\), with respect to \(q\), has the sign of
\[
 (2d-1)-\frac{2d}{1+\log(q/r_0)}-\frac{\log R}{q}.
\]
This expression is increasing in \(q\). Thus the maximum on
\([\max\{2,r_0\},r]\), if this interval is nonempty,
is attained at an endpoint. Its left endpoint contributes at
most \(C_d\sqrt R\): if \(r_0\geq2\), then
\(\rho_L(r_0)=1\); if \(r_0<2\), then
\(\rho_L(2)\leq2^{2d-1}\), since \(L\geq1\).
It follows that
\begin{equation}\label{eq:codegree-moment-endpoints}
 \max_{2\leq q\leq r}R^{1/q}(1+\rho_L(q))
 \leq C_d\left[\sqrt R+\rho_L(r)R^{1/r}\right].
\end{equation}
For even \(r\geq4\), weighted Young's inequality gives
\[
 \rho_L(r)R^{1/r}
 \leq R^{1/2}+\rho_L(r)^{r/(r-2)}
 \leq R^{1/2}+C_d\rho_L(r).
\]
For the last inequality, if \(\rho_L(r)\leq1\), raising it
to a power larger than one only decreases it. Otherwise use
\(\rho_L(r)\leq r^{2d-1}\) and
\(\sup_{r\geq4}r^{(4d-2)/(r-2)}<\infty\).
The case \(r=2\) follows directly from
\eqref{eq:codegree-square-moment} and independence.
Since \(r\rho_L(r)=\Gamma_d(L,r)\), division by \(M\)
proves
\begin{equation}\label{eq:codegree-profile-average}
 \left\|\frac1M\sum_{m=1}^M(Z_m^2-1)\right\|_{L_r}
 \leq C_d\left[
       \sqrt{\frac{r}{M}}+\frac{\Gamma_d(L,r)}M\right],
 \qquad M\geq r.
\end{equation}

\begin{proof}[Proof of \cref{thm:codegree-sketch}]
We have
\(\norm{\Phi A}_2^2/\norm A_{\mathrm F}^2
 =M^{-1}\sum_mZ_m^2\).
Choose \(r=2\lceil\frac12\log(e^2/\delta)\rceil\).
Then \(\log(e^2/\delta)\leq r\leq2\log(e^2/\delta)\).
Since the denominator in \eqref{eq:codegree-heavy-term}
is nondecreasing,
\(\Gamma_d(L_A,r)\leq2^{2d}
\Gamma_d(L_A,\log(e^2/\delta))\).
By enlarging \(C_d\) in \eqref{eq:codegree-rows},
that condition implies \(M\geq r\) and makes the
right-hand side of \eqref{eq:codegree-profile-average}
at most \(\varepsilon/e\).
Markov's inequality bounds the failure probability by
\(e^{-r}\leq\delta/e^2\leq\delta\), proving
\eqref{eq:codegree-norm-guarantee}.
\end{proof}

\begin{proof}[Proof of \cref{cor:relation-kernel}]
For a row index \(m\), write
\[
 Z_{m,i}=\frac{\sqrt M(\Phi A_i)_m}
                         {\norm{A_i}_{\mathrm F}},
 \qquad
 Y_{m,ij}=Z_{m,i}Z_{m,j}-K(A_i,A_j).
\]
The isotropy identity \eqref{eq:khatri-isotropy} gives
\(\E[Z_{m,i}Z_{m,j}]=K(A_i,A_j)\). Although different
data points share the same random row, H\"older's inequality
and \eqref{eq:codegree-row-moment} imply
\[
 \begin{aligned}
 \norm{Y_{m,ij}}_{L_q}
 \leq\norm{Z_{m,i}}_{L_{2q}}
       \norm{Z_{m,j}}_{L_{2q}}+1\leq C_dq\left[1+
             \max_{a\in[N_{\rm dat}]}\rho_{L_{A_a}}(q)\right].
 \end{aligned}
\]
\appreview{The variables \(Y_{m,ij}\) are independent over \(m\).
Choose \(r=2\lceil\frac12\log(e^2N_{\rm dat}^2/\delta)\rceil\).
Interchanging the maximum over \(a\) with the moment-order maximum,
the averaging argument leading to \eqref{eq:codegree-profile-average}
gives, for \(M\geq r\),}
\[
 \left\|\frac1M\sum_{m=1}^MY_{m,ij}\right\|_{L_{r}}
 \leq C_d\left[\sqrt{\frac{r}M}
       +\frac{\max_a\Gamma_d(L_{A_a},r)}M\right].
\]
By the preceding rounding estimate, enlarging \(C_d\) in
\eqref{eq:relation-kernel-rows} ensures \(M\geq r\) and makes
the bound at most \(\varepsilon/e\). Markov's inequality and
a union bound over the ordered pairs give total failure probability
at most \(N_{\rm dat}^2e^{-r}\leq\delta\), proving
\eqref{eq:relation-kernel-guarantee}.
\end{proof}

\subsection{The planted relation model}
\begin{proof}[Proof of \cref{prop:planted-relation}]
\appreview{We first verify the support and mass bounds, treating the
planted box and background separately.} For \(k=d-|I|\in\{1,\ldots,d-1\}\), a background
slice has at most \(n^k\) possible entries and expected cardinality
at most
\[
 \frac{n^k}{8\tau^{d-1}}=\frac{(n/\tau)^k}{8\tau^{d-1-k}}
       \leq\frac{(n/\tau)^k}{8}.
\]
A binomial Chernoff bound puts its cardinality below
\((n/\tau)^k/2\), except with probability
\(e^{-c(n/\tau)^k}\). There are at most \(2^dn^{d-1}\) proper
\appreview{slices. Since \(\log n/(n/\tau)\leq2\log\tau/\tau\)
for \(n\geq\tau^2\) and \(\tau\geq e\), a union bound shows that
all background slices have weighted cardinality at most \(1/2\),
except with probability \(e^{-cn/\tau}\), for sufficiently large
\(\tau\) depending only on \(d\).}
The planted box adds at most \((pb)^k\leq8^{-k}\leq1/8\).
The total background count has mean
\(\mu=(n^d-b^d)/(8\tau^{d-1})\), with
\[
 p^d\mu=\frac{\tau}{8}-\frac{(pb)^d}{8\tau^{d-1}},\qquad
 \mu\asymp_d\frac{n^d}{\tau^{d-1}}.
\]
Thus another Chernoff bound, together with \((pb)^d\leq8^{-d}\),
gives \(\tau/16\leq L\leq \tau/4\) except with probability
\(2\exp\{-cn^d/\tau^{d-1}\}\). This is absorbed by the
stated failure bound, proving \eqref{eq:planted-good-event}.

\appreview{Fix a realization of \(A\) satisfying
\eqref{eq:planted-good-event}. All remaining probabilities are over
the sketch.}

\appreview{For the upper bound, apply \cref{thm:codegree-sketch}
at failure probability \(\delta\), using}
\(L\asymp \tau\asymp_d[\log(e^2/\delta)]^d\). Since \(d\geq2\),
\[
 \Gamma_d\!\left(L,\log\frac{e^2}{\delta}\right)
       \asymp_d\frac{\tau}{(\log \tau)^{2d}},
 \qquad
 \log\frac{e^2}{\delta}
       =o\!\left(\frac{\tau}{(\log \tau)^{2d}}\right).
\]
At error \(1/2\), \eqref{eq:codegree-rows} therefore gives
the asserted upper bound, uniformly in \(n\geq \tau^2\).

\appreview{For the lower bound, it suffices that the first squared
normalized measurement exceeds \(3M/2\). Since the remaining squares
are nonnegative, this forces a relative squared-norm error greater
than \(1/2\).} Set \(T=[b]\), so the planted
box is \(T^d\). Let \(E_k\) be the event that exactly \(k\)
variables in \(T\) equal \(+1\) in each mode, with all other
variables in \(T\) zero. For large \(\tau\),
\(1/16\leq bp\leq1/8\) and \(p\leq1/\tau\). Hence, whenever
\(1\leq k\leq b/2\),
\[
 \Pp_\Phi(E_k)
 =\left[\binom bk(p/2)^k(1-p)^{b-k}\right]^d
 \geq\exp\{-C_dk\log(ek)\}.
\]
Let \(\mathcal F_T\) reveal all factor variables with indices
in \(T\). On \(E_k\), the unnormalized first row is
\(S_A=k^d+U\). Every term of \(U\) contains an outside
variable, so \(\E(U\mid\mathcal F_T)=0\).
Decompose \(U\) according to the nonempty set \(J\) of modes
whose coordinates lie outside \(T\). Fix such a set with
\(j=|J|<d\). Conditional on \(\mathcal F_T\), on the event
\(E_k\) its contribution has the form
\[
 U_J=\sum_{\bm i_J\in([n]\setminus T)^J}
       c_{\bm i_J}\prod_{s\in J}\xi_{1i_s}^{(s)}.
\]
The coefficient \(c_{\bm i_J}\) counts the selected inside tuples
that extend \(\bm i_J\) to a nonzero entry of \(A\). Hence
\[
 0\leq c_{\bm i_J}\leq k^{d-j},\qquad
 \sum_{\bm i_J}c_{\bm i_J}\leq k^{d-j}p^{-j}.
\]
The second inequality applies the slice-count condition to each of
the \(k^{d-j}\) selected inside tuples. Distinct outside monomials
are orthogonal, and therefore
\[
 \E(U_J^2\mid\mathcal F_T)
 =p^j\sum_{\bm i_J}c_{\bm i_J}^2
 \leq p^j\max_{\bm i_J}c_{\bm i_J}
              \sum_{\bm i_J}c_{\bm i_J}
 \leq k^{2(d-j)}.
\]
Contributions from distinct sets \(J\) are also orthogonal. The
part with every mode outside \(T\) has conditional variance at
most \(L\). Consequently, on \(E_k\),
\[
 \operatorname{Var}(U\mid\mathcal F_T)
 \leq L+\sum_{j=1}^{d-1}\binom dj k^{2(d-j)}
 \leq L+2^dk^{2d-2}.
\]

Set \(k=\lceil(64LM)^{1/(2d)}\rceil\).
\appreview{On \(E_k\), conditional Chebyshev gives}
{\color{black}
\[
 \Pp_\Phi\left(U<-k^d/2\mid\mathcal F_T\right)
 \leq \frac{4L}{k^{2d}}+\frac{2^{d+2}}{k^2}
 \leq \frac1{16M}+\frac{2^{d+2}}{k^2}.
\]
}
\appreview{Since \(k^2\geq c_d\tau^{1/d}\), this is at most
\(1/2\) for sufficiently large \(\tau\). Thus
\(\Pp_\Phi\{S_A\geq k^d/2\mid E_k\}\geq1/2\).}
On this event,
\[
 \frac{\norm{\Phi A}_2^2}{\norm A_{\mathrm F}^2}
                  \geq\frac{S_A^2}{ML}
                  \geq\frac{k^{2d}}{4ML}\geq16.
\]
If \(M\leq c_d\tau/(\log \tau)^{2d}\), choosing \(c_d\) small
enough gives
\[
 k\leq C_dc_d^{1/(2d)}\tau^{1/d}/\log \tau+1\leq b/2,
 \qquad C_dk\log(ek)\leq \tau^{1/d}/2\leq\tfrac12\log(e^2/\delta).
\]
Here \(b\geq n/(16\tau)\geq \tau/16\), so the constants are
independent of \(n\). In fact, we obtain the stronger estimate
\begin{equation}
 \Pp_\Phi\!\left(
   \left|\frac{\norm{\Phi A}_2^2}{\norm A_{\mathrm F}^2}-1\right|
       >\tfrac12\right)
 \geq\frac{\sqrt\delta}{2e}>\delta,
\end{equation}
where the strict inequality holds for sufficiently large \(\tau\),
since \(\log(e^2/\delta)\geq \tau^{1/d}>4+2\log2\). This proves the lower bound.
\end{proof}

\begin{proof}[Proof of \cref{cor:planted-polynomial-confidence}]
For large \(n\), the integer
\(\tau=\lfloor[\log(e^2n^a)]^d\rfloor\) satisfies
\(n\geq \tau^2\) and
\(\tau^{1/d}\leq\log(e^2n^a)\leq2\tau^{1/d}\).
Apply \cref{prop:planted-relation} with \(\delta=n^{-a}\).
Its sufficient and necessary row counts have the order in
\eqref{eq:planted-polynomial-rows}, since
\(\tau\asymp_{d,a}(\log n)^d\) and
\(\log \tau\asymp_{d,a}\log\log n\).
\end{proof}

\subsection{The comparison bounds}
\label{sec:relation-baselines}
We give the derivation of \eqref{eq:relation-uncorrected-rows} so
that the averaging step is identical in the two bounds being
compared. \appreview{Rescale the tensor so that}
\(\norm A_\infty=1\), and set \(L_A=p^d\norm A_{\mathrm F}^2\).
Let \(D_I\) be the largest
number of support entries consistent with fixed coordinates in
the nonempty proper mode set \(I\). Assume
\(p^{d-|I|}D_I\leq C_0\) for all such \(I\), and \(L_A\geq1\).

The variables in \(S_A\) are centrally moment bounded with
parameter one. The smoothness parameters in
Schudy and Sviridenko~\cite[Theorem~1.3]{schudy2011bernstein} satisfy
\[
 \operatorname{Var}(S_A)=L_A,\qquad
 \mu_j\leq\max_{|I|=j}p^{d-j}D_I\leq C_0
       \quad(1\leq j<d),\qquad \mu_d=1.
\]
A set of polynomial variables with two indices from the same mode
belongs to no monomial and contributes zero to the maximum defining
\(\mu_j\). Integrating that theorem's tail bound yields
\[
 \norm{S_A}_{L_r}\leq C_{d,C_0}
          \bigl(\sqrt{rL_A}+r^d\bigr),\qquad r\geq2.
\]
The same row estimate follows from the uncorrected partition bound,
using the codegree estimates in the proof of the new theorem.
For \(Z=S_A/\sqrt{L_A}\) and \(Y=Z^2-1\), it follows that
\begin{equation}\label{eq:relation-uncorrected-moments}
 \norm Y_{L_q}\leq C_{d,C_0}
                \left(q+\frac{q^{2d}}{L_A}\right),\qquad q\geq2.
\end{equation}

Let \(r\geq2\) be even and \(M\geq r\). Put \(R=M/r\).
By \cref{lem:moment-averaging},
\[
 \left\|\frac1M\sum_{m=1}^M Y_m\right\|_{L_r}
 \leq \frac{C_{d,C_0}}{R}
       \max_{2\leq q\leq r}
          R^{1/q}\left(1+\frac{q^{2d-1}}{L_A}\right).
\]
\appreview{The endpoint argument used in
\eqref{eq:codegree-moment-endpoints} applies here with
\(\rho_L(q)\) replaced by \(q^{2d-1}/L_A\).
Using \(L_A\geq1\), the same Young inequality gives, for even
\(r\geq4\),}
\[
 \left\|\frac1M\sum_{m=1}^M Y_m\right\|_{L_r}
 \leq C_{d,C_0}\left(\sqrt{\frac{r}{M}}
                         +\frac{r^{2d}}{M L_A}\right).
\]
\appreview{The case \(r=2\) follows directly from
\eqref{eq:relation-uncorrected-moments}.}
Rounding the confidence order up to the next even integer as
in the proof of \cref{thm:codegree-sketch}, and then applying
Markov's inequality at a constant multiple of this bound, proves
\eqref{eq:relation-uncorrected-rows}. For overlap-kernel entries,
H\"older's inequality gives the same centered-product moment curve
whenever the two weighted masses are at least \(L\).
The union bound changes only the chosen confidence order \(r\).

\appreview{For the distribution-only comparison, let
\(Z=\sqrt M(\Phi A)_1/\norm A_{\mathrm F}\), so \(\E Z^2=1\).
The normalized factor \(\xi/\sqrt p\) has sub-Gaussian norm
at most \(C\theta_p\), where
\(\theta_p^2=[p\log(1+1/p)]^{-1}\). Iterating the sub-Gaussian
moment bound over the modes gives
\(\norm Z_{L_r}\leq C_d\theta_p^d r^{d/2}\).} Hence
\(\norm{Z^2-1}_{L_q}\leq C_d\theta_p^{2d}q^d\).
The same endpoint calculation, now with exponent \(d-1\), gives
\[
 \left\|\frac1M\sum_{m=1}^M(Z_m^2-1)\right\|_{L_r}
 \leq C_d\theta_p^{2d}
        \left(\sqrt{\frac{r}{M}}+\frac{r^d}{M}\right).
\]
\appreview{Taking \(r\) at the confidence level and applying Markov's
inequality gives the stated distribution-only row bound.}

\appendix
\section{Moment bounds beyond the bounded decoupled setting}
\label{app:extensions}

\subsection{Bounded non-decoupled chaoses}
\label{app:nondecoupled}

For a diagonal-free polynomial, decoupling replaces the repeated
input vector \(X\) by independent copies
\(X^{(1)},\ldots,X^{(d)}\), at a cost depending only on the degree.
Throughout this subsection, evaluate \(\cS_{I,\cJ}(A)\),
\(B_k(A;r)\), and \(\lambda_A(r)\) with \(p_i^{(s)}=p_i\).

\begin{proposition}[Bounded non-decoupled moment bound]
\label{prop:nondecoupled-moment}
Let \(A=(a_{i_1,\ldots,i_d})\) be a \appreview{real diagonal-free}
\(d\)-tensor, i.e., \(a_{i_1,\ldots,i_d}=0\) whenever two indices
coincide.  Let \(X_1,\ldots,X_n\) be independent, with
\[
 \E X_i=0,\qquad |X_i|\leq1\ \text{a.s.},\qquad
 \E X_i^2\leq p_i,\qquad 0<p_i\leq1,
\]
and define
\[
 Q_A(X)\coloneqq\sum_{i_1,\ldots,i_d=1}^n
              a_{i_1,\ldots,i_d}X_{i_1}\cdots X_{i_d}.
\]
Writing \(m_A\coloneqq\norm A_\infty\), for every \(r\geq2\) we have
\begin{equation}
 \norm{Q_A(X)}_{L_r}
 \lesssim_d
 \max_{1\leq k\leq d}B_k(A;r)
 +\frac{r^dm_A}{\ell(\lambda_A(r))^d}.
\end{equation}
\appreview{When \(\lambda_A(r)>1\), the first term is absorbed
into the second, as in \eqref{eq:effective-main-moment}.}
\end{proposition}

\begin{proof}
Let \(X^{(1)},\ldots,X^{(d)}\) be independent copies of \(X\).
By the decoupling inequality \eqref{eq:decoupling-comparison},
\[
 \norm{Q_A(X)}_{L_r}
 \lesssim_d\norm{Z_A(X^{(1)},\ldots,X^{(d)})}_{L_r}.
\]
\appreview{The copies retain the variance bounds \(p_i\), so
\cref{thm:main} and \eqref{eq:effective-main-moment} give both claims.
Forward decoupling requires no symmetry of \(A\).}
\end{proof}

For \(d=2\), this is the quadratic form \(X^{\mathsf T}AX\)
with zero diagonal, as used in \cref{cor:quadratic-moment}.

\subsection{Sparse \texorpdfstring{$\psi_\alpha$}{psi-alpha} inputs}
\label{app:subexponential}

Let \(1<\alpha\leq\infty\), with \(1/\infty\coloneqq0\), and put
\(\beta\coloneqq1-1/\alpha\in(0,1]\).  For finite \(\alpha\), use the
standard Orlicz norm
\[
 \norm{\chi}_{\psi_\alpha}
 \coloneqq\inf\left\{K>0:
       \E\exp\bigl((|\chi|/K)^\alpha\bigr)\leq2\right\};
 \qquad
 \norm{\chi}_{\psi_\infty}\coloneqq\norm{\chi}_{L_\infty}.
\]
The cases \(\alpha=2\) and \(\alpha=\infty\) correspond to
sub-Gaussian and bounded amplitudes.

\begin{proposition}[Bennett-type bound for sparse sub-exponential chaoses]
\label{prop:alpha-bennett}
Let \(K>0\) and \(X_i^{(s)}\coloneqq\delta_i^{(s)}\chi_i^{(s)}\), where all
Bernoulli random variables and amplitudes are mutually independent and
\[
 \delta_i^{(s)}\sim\Ber(p_i^{(s)}),\qquad 0<p_i^{(s)}\leq1,\qquad
 \E\chi_i^{(s)}=0,\qquad \norm{\chi_i^{(s)}}_{\psi_\alpha}\leq K.
\]
\appreview{For a real rectangular tensor \(A\), evaluate}
\(\cS_{I,\cJ}(A)\), \(B_k(A;r)\), and \(\lambda_A(r)\) using the
Bernoulli probabilities, and write \(m_A\coloneqq\norm A_\infty\).
Then, for every \(r\geq2\),
\begin{equation}
 \norm{Z_A(X)}_{L_r}
 \lesssim_d K^d
 \left[\max_{1\leq k\leq d}B_k(A;r)
       +\frac{r^dm_A}{\ell(\lambda_A(r))^{d\beta}}\right].
\end{equation}
When \(\lambda_A(r)>1\),
\begin{equation}
 \norm{Z_A(X)}_{L_r}
 \lesssim_d\frac{K^dr^dm_A}{\ell(\lambda_A(r))^{d\beta}}.
\end{equation}
The constants are uniform over \(1<\alpha\leq\infty\).
\appreview{The same bounds hold for \(Q_A(X)\) with diagonal-free
\(A\) and \(X_i\coloneqq\delta_i\chi_i\), where the selectors and amplitudes
are mutually independent, \(\delta_i\sim\Ber(p_i)\),
\(\E\chi_i=0\), and \(\norm{\chi_i}_{\psi_\alpha}\leq K\).
Use \(p_i^{(s)}=p_i\) in the weighted norms.}
\end{proposition}

\begin{proof}
By homogeneity, first take \(K=1\).  For finite \(\alpha>1\)
and integers \(m\geq2\), the Orlicz assumption gives
\(\E|\chi_i^{(s)}|^{\alpha m}\leq2m!\). \appreview{Monotonicity of moment norms gives}
\begin{equation}
 \E|X_i^{(s)}|^m
 =p_i^{(s)}\E|\chi_i^{(s)}|^m
 \leq p_i^{(s)}(2m!)^{1/\alpha}
 \leq2p_i^{(s)}(m!)^{1-\beta}.
\end{equation}
For \(\alpha=\infty\), boundedness gives
\(\E|X_i^{(s)}|^m\leq p_i^{(s)}\), so the same bound holds.

\appreview{Use the symmetrized inputs \(V=(X-X')/2\) and independent
replica sums \(S_i^{(s)}=\sum_{u=1}^N Y_{i,u}^{(s)}\) from
\cref{lem:replicated-moment-comparison}, with integer \(N\geq1\).
Convexity gives \(\E|V_i^{(s)}|^m\leq\E|X_i^{(s)}|^m\), while
\(\E(S_i^{(s)})^m\geq Np_i^{(s)}m!\) for even \(m\geq2\).}
Consequently,
\[
 \frac{\E|V_i^{(s)}|^m}{\E(S_i^{(s)})^m}
 \leq\frac{2}{N(m!)^\beta}
 \leq\frac{2e^\beta}{\ell(N)^{m\beta}}
 \leq\left(\frac{\sqrt{2e}}{\ell(N)^\beta}\right)^m,
\]
where
\(N(m!)^\beta\geq(Nm!)^\beta
 \geq e^{-\beta}\ell(N)^{m\beta}\).
\appreview{The symmetrization and coordinate-replacement argument of
\cref{lem:replicated-moment-comparison}, followed by
\cref{lem:sparse-exponential-chaos,lem:norm-scaling}, therefore gives,
after restoring \(K\),}
\begin{equation}\label{eq:alpha-replication-moment}
 \norm{Z_A(X)}_{L_r}
 \lesssim_d K^d
 \frac{\displaystyle\max_{I\subseteq[d]}
              N^{|I^c|/2}\Delta_I(A;r)}
      {\ell(N)^{d\beta}}.
\end{equation}
\appreview{Rounding as in \cref{sec:proof-completion} extends this
estimate to real \(r\geq2\) and \(N\geq1\). Choosing
\(N=\lambda_A(r)\) as in \cref{thm:main} proves both bounds.
The constants are uniform in \(\alpha\), since the scalar comparison
uses \(\sqrt{2e}\) and \(0<\beta\leq1\).
For \(Q_A(X)\), first apply \eqref{eq:decoupling-comparison};
the copies retain \(p_i\) and \(K\).}
\end{proof}

\begin{remark}[Bounded inputs without a Bernoulli factor]
\label{rem:no-bernoulli}
\appreview{A Bernoulli representation is unnecessary when the same
scalar moment bounds hold. Independent centered inputs satisfying}
\[
 |X_i^{(s)}|\leq K\quad\text{a.s.},\qquad
 \E(X_i^{(s)})^2\leq K^2p_i^{(s)}
\]
obey
\(\E|X_i^{(s)}|^m\leq K^{m-2}\E(X_i^{(s)})^2
 \leq K^mp_i^{(s)}\).
The proof therefore applies with \(\beta=1\), recovering
\cref{thm:main} when \(K=1\).
\end{remark}

\appreview{The moment bound also yields a tail estimate in which the
logarithmic gain has power \(\beta\).}

\begin{proposition}[Tail bound in the logarithmic regime]
\label{prop:alpha-tail}
Let \(\eta_A\) be defined by \eqref{eq:explicit-tail-separation},
using the Bernoulli probabilities.  Under the assumptions of
\cref{prop:alpha-bennett}, \appreview{for \(t>0\) with \(\eta_A(t/K^d)\geq e\),}
\begin{equation}
 \Pp\bigl(|Z_A(X)|\geq t\bigr)
 \leq2\exp\left[-c_d
       \left(\frac{t}{K^dm_A}\right)^{1/d}
       \ell\bigl(\eta_A(t/K^d)\bigr)^\beta\right].
\end{equation}
The same bound holds for \(Q_A(X)\) under the non-decoupled
assumptions.  Thus the logarithm enters the tail rate with power
\(1/2\) for sub-Gaussian amplitudes and power one for bounded
amplitudes.
\end{proposition}

\begin{proof}[Proof of \cref{prop:alpha-tail}]
We adapt the proof of \cref{cor:main-tail}, putting
\[
 w=\left(\frac{t}{K^dm_A}\right)^{1/d},\qquad
 N=\eta_A(t/K^d),\qquad \Lambda=\ell(N).
\]
Since \(N>1\), its definition gives, for every \(I,\cJ\),
\[
 K^dN^{|I^c|/2}\cS_{I,\cJ}(A)w^{\kappa_{I,\cJ}}\leq t,
\]
\appreview{Choose \(r=cw\Lambda^\beta\), with \(0<c\leq1\) depending
only on \(d\). If \(r\geq2\), \eqref{eq:alpha-replication-moment} gives}
\[
 \norm{Z_A(X)}_{L_r}
 \leq C_dt\max_{I,\cJ}
       c^{\kappa_{I,\cJ}}\Lambda^{\beta(\kappa_{I,\cJ}-d)}
 \leq C_d\sqrt c\,t,
\]
because \(1/2\leq\kappa_{I,\cJ}\leq d\).
\appreview{For sufficiently small \(c\), Markov's inequality gives
the claimed rate; when \(r<2\), decrease \(c_d\) to make the bound
trivial. For \(Q_A(X)\), apply \eqref{eq:decoupling-comparison} first
and absorb its constant into \(c\), leaving \(\eta_A(t/K^d)\) unchanged.}
\end{proof}

\section{\appreview{General embeddings}}
\label{sec:general-sketches}
The support-count condition gives an explicit row bound for a class
of tensors. For arbitrary fixed tensors, we use the weighted
partition norms in the row moment curve. \rev{We give this version below,
together with its finite-set and subspace consequences.} Throughout this \appreview{appendix},
\(n\) and \(p\) are arbitrary and the sketch is independent of the
input data.

\appreview{We first define a bound for the moments of one normalized
measurement. For \(x\neq0\) and \(r\geq2\), set}
\begin{equation}\label{eq:khatri-standard-norm-factor}
 \mathcal R_x(r)\coloneqq1+\min\left\{
 \left(\frac rp\right)^{d/2},\;
 \frac1{p^{d/2}\norm{x}_{\mathrm F}}
 \inf_{N\geq1}
 \frac{\max_{I\subseteq[d]}N^{|I^c|/2}\Delta_I(x;r)}
      {\ell(N)^d}\right\}.
\end{equation}
Here \(\Delta_I\) is defined in \eqref{eq:Delta} with all
second-moment bounds equal to \(p\). The first expression follows
from Khintchine's inequality; the second includes the weighted slice
norms and the logarithmic correction. Up to a constant depending only on \(d\), each expression bounds
\(\norm{\sqrt M(\Phi x)_1}_{L_r}/\norm x_{\mathrm F}\), as shown below.
\appreview{To control the error after averaging \(M\) squared
measurements, for an even integer \(r\geq2\), set}
\begin{equation}
 \mathcal E_x(M,r)\coloneqq\frac1M\max_{q\in\{2,\ldots,r\}}
       \frac rq\left(\frac Mr\right)^{1/q}\mathcal R_x(2q)^2.
\end{equation}
The expression \(\mathcal E_x(M,r)\) results from applying
\cref{lem:moment-averaging} to the centered normalized squared
measurements \(M(\Phi x)_m^2/\norm{x}_{\mathrm F}^2-1\), whose
\(L_q\)-norms are bounded by \(C_d\rev{\mathcal R_x}(2q)^2\).
Each term decreases with \(M\), so the sufficient conditions below
can be checked by increasing the row count. The replication parameter \(N\) in
\eqref{eq:khatri-standard-norm-factor} may be chosen separately at
each moment order. In particular, no single upper bound on
\(\mathcal R_x(r)/\sqrt r\) is needed over the whole range.

\begin{proposition}[Sparse Khatri--Rao embedding]
\label{prop:sparse-khatri-rao}
There is a constant \(C_d\) with the following properties.
Fix \(0<\varepsilon\leq1\) and \(0<\delta<1\).
\begin{enumerate}
\item For every fixed \(x\neq0\) and every integer \(M\geq1\),
\begin{equation}\label{eq:khatri-fixed-vector}
 \Pp\left(
 \left|\frac{\norm{\Phi x}_2^2}{\norm{x}_{\mathrm F}^2}-1\right|
 > C_d\rev{\mathcal E_x}\!\left(M,
       2\left\lceil\frac12\log\frac{e^2}{\delta}\right\rceil\right)
       \right)\leq\delta.
\end{equation}
Consequently,
\begin{equation}
 C_d\rev{\mathcal E_x}\!\left(M,
       2\left\lceil\frac12\log\frac{e^2}{\delta}\right\rceil\right)
       \leq\varepsilon
\end{equation}
is sufficient for relative squared-norm error at most
\(\varepsilon\), with failure probability at most \(\delta\).
\item For a fixed set \(\mathcal X=\{x_1,\ldots,x_m\}\) of
distinct tensors, with \(m\geq2\), the condition
\[
 C_d\max_{i<j}\mathcal E_{x_i-x_j}\!\left(M,
       2\left\lceil\frac12\log\frac{e^2m^2}{\delta}\right\rceil\right)
       \leq\varepsilon
\]
implies, with probability at least \(1-\delta\),
\[
 (1-\varepsilon)\norm{x_i-x_j}_{\mathrm F}^2
 \leq\norm{\Phi x_i-\Phi x_j}_2^2
 \leq(1+\varepsilon)\norm{x_i-x_j}_{\mathrm F}^2
 \quad\text{for all }i<j.
\]
\item For a fixed \(k\)-dimensional subspace \(W\), with
\(k\geq1\), the condition
\[
 C_d\sup_{x\in W\setminus\{0\}}
      \mathcal E_x\!\left(M,
       2\left\lceil\frac12\log\frac{e^2 9^k}{\delta}\right\rceil\right)
       \leq\varepsilon
\]
implies, with probability at least \(1-\delta\),
\[
 (1-\varepsilon)\norm{x}_{\mathrm F}^2
 \leq\norm{\Phi x}_2^2
 \leq(1+\varepsilon)\norm{x}_{\mathrm F}^2
 \quad\text{for all }x\in W.
\]
\end{enumerate}
\end{proposition}
The finite-set condition involves pairwise differences. The subspace
condition is uniform over \(W\), and therefore also covers vectors
chosen after the sketch has been observed.

\begin{proof}
Put
\[
 Z_m(x)=\sqrt M(\Phi x)_m
 =p^{-d/2}\sum_{\bm i}x_{\bm i}
                  \prod_{s=1}^d\xi_{mi_s}^{(s)}.
\]
Independence and centering give
\(\E[Z_m(x)Z_m(y)]=\langle x,y\rangle\). Applying
\cref{prop:replication-profile} for every \(N\geq1\) gives
the second bound in \eqref{eq:khatri-standard-norm-factor}.
For the first, write
\(\xi_{mi}^{(s)}=b_{mi}^{(s)}\sigma_{mi}^{(s)}\), where the
\(b_{mi}^{(s)}\) are independent \(\Ber(p)\) variables and the
\(\sigma_{mi}^{(s)}\) are independent Rademacher variables,
independent of the selectors. Conditional on the selectors,
repeated Khintchine inequalities and Minkowski's inequality give
\[
 \norm{Z_m(x)}_{L_r(\mathrm{signs})}
 \leq p^{-d/2}r^{d/2}
 \left(\sum_{\bm i}x_{\bm i}^2
             \prod_{s=1}^d b_{mi_s}^{(s)}\right)^{1/2}
 \leq p^{-d/2}r^{d/2}\norm{x}_{\mathrm F}.
\]
This conditional bound is uniform in the Bernoulli variables.
Combining the two estimates yields
\begin{equation}
 \norm{Z_m(x)}_{L_r}\leq C_d\rev{\mathcal R_x}(r)\norm{x}_{\mathrm F}.
\end{equation}
Consequently, since \(\mathcal R_x(r)\geq1\),
\begin{equation}
 \norm{Z_m(x)^2-\norm{x}_{\mathrm F}^2}_{L_q}
 \leq C_d\rev{\mathcal R_x}(2q)^2\norm{x}_{\mathrm F}^2.
\end{equation}
For every even integer \(r\geq2\),
\cref{lem:moment-averaging} and Markov's inequality bound the
probability that the relative squared-norm error exceeds
\(C_d\rev{\mathcal E_x}(M,r)\) by \(e^{-r}\), after absorbing the Markov
factor \(e\) into \(C_d\).
Taking \(r=2\lceil\frac12\log(e^2/\delta)\rceil\)
proves \eqref{eq:khatri-fixed-vector}, since
\(e^{-r}\leq\delta/e^2\leq\delta\).
For the finite-set claim, take
\(r=2\lceil\frac12\log(e^2m^2/\delta)\rceil\)
and use a union bound over fewer than \(m^2\) differences.

For the subspace claim, let \(\mathcal N\) be a \(1/4\)-net
of the unit sphere of \(W\), with \(|\mathcal N|\leq9^k\).
Apply the preceding even-moment bound on this net with
\(r=2\lceil\frac12\log(e^2 9^k/\delta)\rceil\).
For the self-adjoint
operator \(\Theta=(\Phi|_W)^*(\Phi|_W)-I_W\), the usual net
argument gives
\(\norm\Theta_{\mathrm{op}}\leq2\max_{x\in\mathcal N}|\langle\Theta x,x\rangle|\).
Indeed, approximate a unit vector attaining the operator norm by
a net point; the two resulting error terms sum to at most
\(\frac12\norm\Theta_{\mathrm{op}}\). A union bound and an
enlargement of \(C_d\) complete the proof.
\end{proof}

\appreview{We next check that the general bound recovers the standard
row count for dense Rademacher tensor-product maps, where \(p=1\).
In this case, \eqref{eq:khatri-standard-norm-factor} gives}
\(\mathcal R_x(2q)^2\leq C_dq^d\) for every tensor \(x\). For an even integer
\(r\geq2\), if \(M\geq r\), then
\begin{equation}
 \mathcal E_x(M,r)\leq C_d\left(\sqrt{\frac rM}+\frac{r^d}{M}\right).
\end{equation}
\appreview{This follows from the endpoint argument in
\eqref{eq:codegree-moment-endpoints} and the subsequent Young
inequality, now with the power \(q^{d-1}\); the case \(r=2\)
is immediate.} Taking
\(r=2\lceil\frac12\log(e^2/\delta)\rceil\)
and enlarging \(C_d\) to absorb the rounding shows that
\[
 M\geq C_d\left[
      \varepsilon^{-2}\log\frac{e^2}{\delta}
      +\varepsilon^{-1}\left(\log\frac{e^2}{\delta}\right)^d
      \right]
\]
\appreview{is sufficient for relative squared-norm error at most
\(\varepsilon\) for a fixed tensor, with failure probability at most
\(\delta\), matching the dependence in
\cite[Theorem~2]{AhleKnudsen2019}. The first term ensures \(M\geq r\).}

\section*{Acknowledgements}
GPT-5.6 assisted in developing the proof. The authors independently verified the
argument and take full responsibility for the final proof. K.W. is supported by Hong Kong RGC grant GRF 16305526. Y.Z. was partially supported by the Simons Grant MPS-TSM-00013944 and NSF DMS-2606337. This work was carried out while Y.Z. was visiting the Simons Institute for the Theory of Computing during the Spectral Theory Beyond Graphs program in Fall 2026.

%\section*{Declarations}
%\noindent\textbf{Data availability.}
%No datasets were generated or analysed during the current study.

% TODO(authors): add the "Funding" and "Competing interests" statements
% required by Journal of Theoretical Probability / Springer Nature.

\bibliography{Polynomial}
\bibliographystyle{abbrv}

%%%%%%%%%%%%%%%%%%%%

\end{document}